\documentclass{amsart}

\usepackage{mathtools, amsmath, amssymb, amsrefs}
\usepackage{url}
\usepackage{bbm}
\usepackage[inline, shortlabels]{enumitem}
\usepackage[margin=3cm]{geometry}
\usepackage{xcolor}

\DeclarePairedDelimiter\floor{\lfloor}{\rfloor}
\DeclareMathOperator{\sgn}{sgn}

\newcommand{\R}{\mathbb{R}}
\newcommand{\N}{\mathbb{N}}
\newcommand{\C}{\mathbb{C}}
\newcommand{\Z}{\mathbb{Z}}

\newcommand{\D}{\mathbb{D}}
\newcommand{\T}{\mathbb{T}}
\newcommand{\E}{\mathbb{E}}

\renewcommand{\d}{\mathrm{d}}

\renewcommand{\i}{\mathrm{i}}

\renewcommand{\Re}{\operatorname{Re}}
\renewcommand{\Im}{\operatorname{Im}}

\newcommand{\dd}{\mathbf{d}}
\newcommand{\g}{\mathbf{g}}
\newcommand{\h}{\mathbf{h}}

\newtheorem{thm}{Theorem}[section]
\newtheorem{prop}[thm]{Proposition}
\newtheorem{lemma}[thm]{Lemma}
\newtheorem{cor}[thm]{Corollary}

\theoremstyle{definition}

\theoremstyle{remark}
\newtheorem{remark}[thm]{Remark}

\begin{document}

\author[K. Courteaut]{K. Courteaut\textsuperscript{*}}
\author[K. Johansson]{K. Johansson\textsuperscript{$\dagger$}}
\thanks{\textsuperscript{*}{Department of Mathematics, KTH Royal Institute of Technology, kc5733@nyu.edu}}
\thanks{\textsuperscript{$\dagger$}Department of Mathematics, KTH Royal Institute of Technology, kurtj@kth.se}
\thanks{Supported by the grant KAW 2015.0270 from the Knut and Alice Wallenberg
Foundation}
\title{Partition function for the 2d Coulomb gas on a Jordan curve}

\maketitle
\begin{abstract}
We prove an asymptotic formula for the partition function of a 2d Coulomb gas at inverse temperature $\beta>0$, confined to lie on a Jordan curve. The partition function can include a linear statistic. The asymptotic formula involves a Fredholm determinant related to the Loewner energy of the curve, and also an expression involving the 
sampling function, the exterior conformal map for the curve and the Grunsky operator. The asymptotic formula also gives a central limit theorem for linear statistics of the particles in the gas. 
\end{abstract}

%%% Section numbering
%\setcounter{section}{-1}

\section{Introduction}
\label{introduction}

\setcounter{equation}{0}
\setcounter{thm}{0}
Consider a planar Coulomb gas restricted to lie on a Jordan curve $\gamma$ in the complex plane. More precisely, let $z_1,\cdots, z_n$ denote the positions of $n$ particles on $\gamma$ with joint density given by
\begin{equation}\label{probbasic}
\d\mu_n^\beta(z_1,\cdots, z_n) = \frac{1}{Z_n^\beta} \prod_{1\leq \mu<\nu \leq n} \lvert z_{\mu}-z_{\nu}\rvert^\beta \prod_{\mu=1}^n|\d z_\mu|.
\end{equation}
Here $Z_n^\beta$ is a normalization constant, the \textit{partition function}, and $\beta>0$ can be thought of as the inverse temperature of the particle system. In the case where $\gamma$ is the unit circle $\T$ we recover the well-studied Circular $\beta$-Ensemble. In this case $z_1,\cdots, z_n$ are the eigenvalues of a random matrix, and in particular we get classical random matrix ensembles COE ($\beta=1$), CUE ($\beta=2$), and CSE ($\beta=4$), see e.g.~ \cite{For}. The case $\beta=2$ for general curves $\gamma$ is also special as it gives rise to a determinantal point process on $\gamma$. Let $g$ be a function defined on $\gamma$. By Andreief's identity,
\[\frac{1}{n!}\int_{\gamma^n} \prod_{1\leq \mu<\nu \leq n} \lvert z_{\mu}-z_{\nu}\rvert^2 \prod_{\mu=1}^n e^{g(z_\mu)} |\d z_\mu|  = \det \Big( \int_{\gamma} z^j\bar{z}^k e^{g(z)} |\d z| \Big)_{j,k=1}^n.  \]
The right-hand side is a generalized Toeplitz determinant, which for real $g$ is related to polynomials on the curve $\gamma$ which are orthogonal w.r.t.\ the weight $e^g$, see Section 16.2 in \cite{Szego}.

In this paper we obtain an asymptotic formula for
\begin{equation}\label{eq:maindef}
D_n^\beta[e^g] \coloneqq \frac{1}{n!} \int_{\gamma^n} \prod_{1\leq \mu<\nu \leq n} \lvert z_{\mu}-z_{\nu}\rvert^\beta \prod_{\mu=1}^n e^{g(z_{\mu})} \lvert d z_{\mu}\rvert, 
\end{equation}
as the number of particles $n$ goes to infinity. In particular, we obtain the asymptotics of the partition function $Z_{n,\beta}(\gamma)=D_n^\beta[1]$.
The case $\beta = 2$ was studied in \cite{Joh88} and \cite{Joh22}, and the latter gave 
an asymptotic formula for the partition function. An asymptotic formula for \eqref{eq:maindef} was conjectured in \cite{Joh22}, and in this paper we prove this conjecture under rather strong regularity assumptions on the curve $\gamma$ and the function $g$. We first generalize to any $\beta>0$ the asymptotic formula for $D_n^\beta[e^{g}]/D_n^\beta[1]$ proved in \cite{Joh88} in the case $\beta=2$, and then use this result to obtain the asymptotics of $D_n^\beta[e^g]$. The asymptotic formula for $D_n^\beta[e^g]$ was predicted also in \cite{WieZab22} via a non-rigorous argument.

The function $z\mapsto D_n^\beta[e^{zg}]/D_n^\beta[1]$ is the Laplace transform of the linear statistic $\sum_\mu g(z_\mu)$. Its limit therefore provides the limiting distribution (and moments) of $\sum_{\mu=1}^n g(z_\mu)$. It is typical that for eigenvalues of random matrices and one-dimensional Coulomb gases, such linear statistics do not need to be normalized by $\sqrt{n}$ in order to converge, unlike in the classical central limit theorem. We will see that this is the case here as well, and that the limit is normal with a mean and variance depending on $g$ and the exterior conformal mapping related to $\gamma$.

In order to state our main results we introduce the following notation. Denote by $\Omega_+$ the unbounded component of the complement of $\gamma$, and by $\Omega_-$ the bounded component. Let $\D$ denote the open unit disc, and $\D^*=\{z\,|\,|z|>1\}$ the exterior of the unit circle. Let $\phi$ be the unique conformal map from $\D^*$ onto $\Omega_+$ such that $\lim_{z\to\infty} \phi(z) = \infty$ and $\lim_{z\to \infty} \phi(z)/z>0$. Then $\lim_{z\to\infty} \phi(z)/z = \mathrm{cap}(\gamma)$, the \textit{capacity} of the curve. By Carath\'{e}odory's theorem, $\phi$ has a continuous one-to-one extension to $\D^c$, so $\phi(e^{\i t})$, $t\in[0,2\pi)$, is a parameterization of $\gamma$.
   
Because $\phi$ is conformal we can write
\begin{equation}\label{Grunsky}
     \log \Big( \frac {\phi(z)-\phi(w)}{z-w} \Big) =  \log (\mathrm{cap}(\gamma))-\sum_{k,l \geq 1} a_{kl} z^{-k} w^{-l}
\end{equation}
if $|z|>1$, $|w|>1$, see \cite{Pom}. The coefficients $\{a_{kl}\}$ are called the \textit{Grunsky coefficients}. Note that $a_{kl}=a_{lk}$. Taking the limit $w\to z$ gives
\begin{equation}\label{Grunsky2}
    \log \phi'(z) = -\sum_{k\geq2} \Big(\sum_{j=1}^{k-1} a_{j,k-j} \Big) z^{-k}, \quad |z|> 1.
\end{equation}
The \textit{Grunsky operator} is the operator on $\ell_2(\C)$ given by 
\[B = \big( \sqrt{kl}a_{kl} \big)_{k,l\geq 1}.\]
We let $B^{(1)}$ and $B^{(2)}$ denote its real and imaginary parts and define the operator $K$ on $\ell_2(\R) \oplus \ell_2(\R)$ by
\begin{equation}\label{eq:K}
    K = \begin{pmatrix}
B^{(1)} & B^{(2)} \\
B^{(2)} & -B^{(1)}
\end{pmatrix}.
\end{equation}
Write
\[
\begin{pmatrix}
    \mathbf{x_\theta} \\ \mathbf{y_\theta}
\end{pmatrix} = \begin{pmatrix}
    (\frac{1}{\sqrt{k}} \cos(k\theta))_{k\geq1} \\ (\frac{1}{\sqrt{k}} \sin(k\theta))_{k\geq1}
\end{pmatrix}.
\]
Then, by \eqref{Grunsky},
\begin{equation}\label{KGrunskyrel}
-\log \Big| \frac{ \phi(e^{\i\theta})-\phi(e^{\i \omega})}{e^{\i\theta}-e^{\i \omega}} \Big| +\log (\mathrm{cap}(\gamma))=\sum_{k,l\geq 1} \Re (a_{kl}) \cos(k \theta+ l \omega)+\Im (a_{kl})\sin(k \theta+ l \omega) = \begin{pmatrix}
    \mathbf{x_\omega} \\ \mathbf{y_\omega}
\end{pmatrix}^t K \begin{pmatrix}
    \mathbf{x_\theta} \\ \mathbf{y_\theta}
\end{pmatrix}, 
\end{equation}
so $K$ appears when we take the real part in \eqref{Grunsky}. The operator $K$ is also related to the Neumann-Poincar\'e operator, see Proposition \ref{prop:NP}.

Given a complex-valued function $g$ on $\gamma$ we can move it to the unit circle using the map $\phi$ and expand it in a Fourier series
\[
g\circ\phi(e^{\i\theta}) = \frac{a_0}{2} +\sum_{k\geq 1} (a_k\cos(k\theta)+ b_k\sin(k\theta)),
\]
where $a_n$ and $b_n$ can be complex-valued.
Write
\begin{equation}\label{gvector}
 \g = \frac{1}{2} \begin{pmatrix}
    (\sqrt{k}a_k)_{k\geq 1} \\ (\sqrt{k}b_k)_{k\geq 1}
\end{pmatrix},
\end{equation}
which is an element in $\ell_2(\C) \oplus \ell_2(\C)$ if and only if $g\circ \phi$ belongs to the Sobolev space $H^{1/2}(\T)$. Then we can write the Fourier series as
\begin{equation}\label{ggvec}
g\circ\phi(e^{\i\theta})=\frac{a_0}2+2\begin{pmatrix}
    \mathbf{x_\theta} \\ \mathbf{y_\theta}
\end{pmatrix}^t\g.
\end{equation}
We also set
\begin{equation}\label{eq:d}
    \dd = \frac{1}{2}\begin{pmatrix}
( \sqrt{k} \sum_{j=1}^{k-1} \Re a_{j,k-j} )_{k\geq1} \\
( \sqrt{k} \sum_{j=1}^{k-1} \Im a_{j,k-j} )_{k\geq1}
\end{pmatrix},
\end{equation}
so that by \eqref{Grunsky2}
\begin{equation}\label{logphiprime}
\log|\phi'(e^{\i\theta})|=-2\begin{pmatrix}
    \mathbf{x_\theta} \\ \mathbf{y_\theta}
\end{pmatrix}^t\dd.
\end{equation}
Furthermore, we let
\begin{equation}\label{gbetavec}
\g_\beta=\g+(\frac{\beta}{2}-1)\dd.
\end{equation}
Finally, in the case when $\gamma$ is the unit circle $\T$ the partition function can be computed explicitely,
\begin{equation}\label{Selberg}
Z_{n,\beta}(\T) = \frac{(2\pi)^n}{n!}\frac{\Gamma(1+\beta n/2)}{\Gamma(1+\beta/2)^n}.
\end{equation}
This follows from Selberg's integral, see \cite{Mehta}[Theorem 12.1.1.].
We prove the following.

\begin{thm}\label{mainthm2}
Assume that $\gamma$ is a $C^{12+\alpha}$ Jordan curve and $g\in C^{4+\epsilon}(\gamma)$, for some $\alpha>0$, $\epsilon>0$. Then, as $n\to\infty$,
\begin{equation}\label{eq:asym}
   D_n^\beta[e^g]  = \frac{Z_{n,\beta}(\T)\mathrm{cap}(\gamma)^{\beta n^2/2+(1-\beta/2)n}}{\sqrt{\det(I+K)}}\exp \Big(n\frac{a_0}{2}+\frac{2}{\beta}  \mathbf{g}_\beta^t (I+K)^{-1}\mathbf{g}_\beta +o(1) \Big).
\end{equation}
\end{thm}
Note that when $\beta=2$ and $\gamma=\T$ the left side of \eqref{eq:asym} is an ordinary Toeplitz determinant and the statement of the theorem is the strong Szeg\H{o} limit theorem.
The strong assumptions on the curve $\gamma$ and the function $g$ come from the techniques used in the proof. In view of the results in \cite{Joh22}, it is natural to conjecture that it should be enough to assume that $K$ is a trace class operator. The eigenvalues of $K$ are plus/minus the singular values of $B$, so that $\det(I+K)=\det(I-BB^*)$, which means that $K$ is trace-class is equivalent to the Grunsky operator being a Hilbert-Schmidt operator. Curves for which $B$ is a Hilbert-Schmidt operator are called Weil-Petersson quasicircles and form an interesting class of curves, see \cite{Bish} and \cite{Wang}. The Fredholm determinant is directly related to the \emph{Loewner energy} $I^L(\gamma)$ of the curve, in fact
$I^L(\gamma)=-12\log\det(I-BB^*)$, and the Loewner energy is finite if and only if $\gamma$ is a Weil-Petersson quasicircle, see \cite{Wang} for more on this.
Concerning the regularity condition on $g$ it can be somewhat relaxed, see Remark \ref{rem:greg}. The optimal condition on $g$ is not clear. That the function $g$ has a finite $H^{1/2}$-norm is not sufficient for \eqref{eq:asym} to hold even in the case when $\gamma$ is the
unit circle, see \cite{Lam} for a counterexample when $\beta=4$.
The fact that $I+K$ in \eqref{eq:asym} is invertible is a consequence of the strengthened Grunsky inequality, see the discussion following Lemma \ref{lemma:grunskyineq}. 
For the relation between the formula in \eqref{eq:asym} and the corresponding formula in \cite{WieZab22} see Remark \ref{rem:NJO}.

Let $\psi=\phi^{-1}$ be the conformal map from $\Omega_+$ to the exterior of the
unit circle. If $\gamma$ is a $C^{k+\alpha}$-curve, $0<\alpha<1$,  then by Kellogg's theorem $\phi$ extends to a bijective $C^{k+\alpha}$-function on $\T$. Hence, if $k\ge 1$, $\psi'$ is well-defined on $\T$. Define $g_\beta$ on $\gamma$ by,
\[
g_\beta\circ\phi(e^{\i\theta})=2\begin{pmatrix}
    \mathbf{x_\theta} \\ \mathbf{y_\theta}
\end{pmatrix}^t
\g_\beta.
\]
We see that
\[
g_\beta(z)=g(z)-\frac{a_0}2+(\frac \beta 2-1)\log|\psi'(z)|, \quad z\in\gamma,
\]
by \eqref{ggvec}, \eqref{logphiprime}, \eqref{gbetavec}, and the fact that $\psi$ is the inverse function of $\phi$.  In the case when $g$ is real-valued, an alternative expression for the quantity in the exponent in \eqref{eq:asym} can be given in terms of Dirichlet energies from the following proposition.

\begin{prop}\label{prop:diffvar}
Assume that $g$ is a real-valued $C^{1+\alpha}$ function on the $C^{4+\alpha}$ Jordan curve $\gamma$. Let $g_{\pm}$ be the (bounded) harmonic extension of $g$ to
$\Omega_\pm$. Then,
\begin{equation}\label{Dirichlet}
\g^t (I+K)^{-1}\g   = \frac{1}{8\pi} \Big( \int_{\Omega_+} |\nabla g_{+}|^2 \d x \d y + \int_{\Omega_-} |\nabla g_{-}|^2 \d x \d y \Big). 
 \end{equation}
\end{prop}
The conditions in the theorem are such that the computations in the proof work easily, but they can be weakened.
If we apply this to the function $g_\beta$ we get the formula
\begin{equation*}
\frac{2}{\beta}  \mathbf{g}_\beta^t (I+K)^{-1}\mathbf{g}_\beta   = \frac{1}{4\pi\beta} \Big( \int_{\Omega_+} |\nabla g_{\beta,+}|^2 \d x \d y + \int_{\Omega_-} |\nabla g_{\beta,-}|^2 \d x \d y \Big), 
 \end{equation*}
 where $g_{\beta,\pm}$ is the (bounded) harmonic extension of $g_\beta$ to $\Omega_\pm$.

Denote by $\E_{\gamma^n}^\beta$ expectation with respect to the probability measure \eqref{probbasic} on $\gamma^n=\gamma\times\dots\times\gamma$ so that
\[
 \E_{\gamma^n}^\beta \big[ \prod_{\mu=1}^n e^{g(z_\mu)} \big]=\frac{D_n^\beta[e^g]}{D_n^\beta[1]}.
 \]
The proof of Theorem \ref{mainthm2} is based on the following \textit{relative Szeg\H{o} type theorem}.

\begin{thm}\label{mainthm}
Assume that $\gamma$ is a $C^{9+\alpha}$ Jordan curve and $g$ a complex-valued $C^{4+\epsilon}$ function on $\gamma$, for some $\alpha>0$, $\epsilon>0$. 
Then, as $n\to\infty$,
\begin{equation}\label{eq:main}
   \E_{\gamma^n}^\beta \big[ \prod_{\mu=1}^n e^{g(z_\mu)} \big] = \exp \Big(n\frac{a_0}{2}+\frac{2}{\beta}  \g^t (I+K)^{-1}\g+2\big(1-\frac{2}{\beta}\big) \dd^t(I+K)^{-1}\g + o(1)\Big).
\end{equation}
\end{thm}

By replacing $g$ with $zg$, $z\in\C$, in \eqref{eq:main} and recognizing the Laplace transform of $\sum_\mu g(z_\mu)$ on the left-hand side of \eqref{eq:main}, we obtain

\begin{cor}
Assume that $\gamma$ is a $C^{9+\alpha}$ Jordan curve, $g\in C^{4+\epsilon}(\gamma)$, for some $\alpha>0$, $\epsilon>0$.
Then, as $n\to\infty$,
\[\sum_\mu g(z_\mu)-n \int_0^{2\pi} g(\phi(e^{\i\theta})) \frac{\d\theta}{2\pi} \overset{d}{\Rightarrow} \mathcal N(\mu_g,\sigma_g^2) \]
where 
\[ \mu_g = 2\big(1-\frac{2}{\beta}\big) \dd^t(I+K)^{-1}\g,\quad \sigma_g^2 = \frac{4}{\beta}  \g^t (I+K)^{-1}\g.\]
\end{cor}
Note that 
\[
\frac {a_0}2=\int_0^{2\pi} g(\phi(e^{\i\theta})) \frac{\d\theta}{2\pi}=\int_\gamma g(z)\,d\nu_{\text{eq}}(z),
\]
is the expectation of $g$ with respect to the equilibrium measure $\nu_{\text{eq}}$ on the curve $\gamma$.

We recall that the case $\gamma=\T$ gives the Circular $\beta$-Ensemble, which can be realized as the eigenvalues of a random matrix constructed in \cite{KilNen}. Results related to Theorem \ref{mainthm} for the C$\beta$E include \cite{JiaMat}, which gives a CLT for polynomials that follows from estimates on the moments, and \cite{Webb}, which generalizes this result and gives a rate of convergence. In \cite{Lam} the analogue of \eqref{eq:main} for $\gamma= \T$ in the mesoscopic regime was obtained, and in \cite{HarLam}, the high temperature regime was considered, still with $\gamma=\T$.

The paper is organized as follows: in Section \ref{sec:prel} we give some preliminary results that will be used in the proofs of the main theorems, and we also prove Proposition \ref{prop:diffvar}. The proof of the relative Szeg\H{o} type theorem, Theorem \ref{mainthm}, is given in Section \ref{sec:relSz}. The last section deals with the asymptotics of the partition function which combined with Theorem \ref{mainthm} will prove Theorem \ref{mainthm2}.

\vskip 0.3cm
\noindent
{\bf Acknowledgement}. We thank Yacin Ameur and Fredrik Viklund for helpful comments on the paper.

\section{Preliminaries and an integral equation}\label{sec:prel}
In this section we will first discuss some preliminary results that we will need. We will also discuss a certain integral equation that will be important in the proof of the relative Szeg\H{o}
theorem, and we will give the proof of Proposition \ref{prop:diffvar}.

\subsection{Preliminaries}
From now on, we assume without loss of generality that $\mathrm{cap}(\gamma)=1$. We can reduce the general case to this one by dividing both sides of \eqref{eq:maindef} by $\mathrm{cap}(\gamma)^{\beta n^2/2+(1-\beta/2)n}$ and replacing $g\circ\phi$ by $g\circ(\phi/\mathrm{cap}(\gamma))$. We also assume that the mean $a_0/2=0$, and that $\gamma$ is  $C^{m+\alpha}$ for some $\alpha>0$, $m\ge 1$. Consequently, by Kellogg's theorem, $\phi$ extends to be $C^{m+\alpha}$ on $\overline{\D^*}$. Note that the Grunsky coefficients are Fourier coefficients of a function on $\T^2$ and hence we can use an integration by parts argument to see that there exists a constant $A$, that only depends on $\gamma$, such that the Grunsky coefficients satisfy
\begin{equation}\label{eq:coeffdecay}
    |a_{kl}|\leq A k^{-p-\alpha/2}l^{-q-\alpha/2}, \quad p+q=m-1,
\end{equation} 
$k,l\ge 1$. For Theorem \ref{mainthm2} we will take $m=12$, and for Theorem \ref{mainthm} we will take $m=9$.
The Grunsky coefficients also satisfies the following inequality
\[\sum_{k\geq 1} \big| \sum_{l\geq 1} \sqrt{kl}a_{kl}w_l \big|^2 \leq \sum_{k\geq 1} |w_k|^2 \]
for any complex sequence $(w_k)_{k\geq 1}\in \ell_2(\C)$. This is known as the \textit{Grunsky inequality}, see e.g.~ \cite{Pom}. We will need the following stronger version.

\begin{lemma}[The strengthened Grunsky inequality]\label{lemma:grunskyineq}
There is a constant $\kappa<1$ such that
\begin{equation}\label{Grunskyineq}
\sum_{k\geq 1} \big| \sum_{l\geq 1} \sqrt{kl}a_{kl}w_l \big|^2 \leq \kappa^2 \sum_{k\geq 1} |w_k|^2 
\end{equation}
for any complex sequence $(w_k)_{k\geq 1}\in \ell_2(\C)$ if and only if $\phi(\D^*)$ is bounded by a quasicircle.
\end{lemma}

A proof can be found in chapter 9.4 of \cite{Pom}. Since we assume that $\gamma$ is a $C^{9+\epsilon}$ curve it is clearly quasiconformal and hence \eqref{Grunskyineq} holds. This implies that the Grunsky operator satisfies $\|B\|\le\kappa<1$, and thus, for any real $\mathbf{x}$, $\mathbf{y}\in l^2(\N)$, 
\[ \big\|K \begin{pmatrix} \mathbf{x} \\ \mathbf{y} \end{pmatrix} \big\|_2^2 = \big\| B(\mathbf{x}-\i \mathbf{y}) \big\|_2^2 \leq \kappa^2 \| \mathbf{x}-\i \mathbf{y} \|_2^2 = \kappa^2 \big\| \begin{pmatrix} \mathbf{x} \\ \mathbf{y} \end{pmatrix} \big\|_2^2  \]
i.e.~ $\| K\| \le \kappa<1$ as well. In particular $I+K$ is invertible.

Thanks to the following lemma, we can assume that $g$ is real-valued when we prove Theorem \ref{mainthm}.

\begin{lemma}\label{lemma:real}
If \eqref{eq:main} holds for any real-valued function $g\in C^{4+\epsilon}(\gamma)$, then \eqref{eq:main} holds for any complex-valued  $g\in C^{4+\epsilon}(\gamma)$, $\epsilon>0$. The limit is the same but the Fourier coefficients $a_n$ and $b_n$, given by the usual formulas, are now complex numbers.
\end{lemma}

\begin{proof}
Assume that $g\in C^{4+\epsilon}(\gamma)$ is complex-valued. Define the analytic functions
 \[f_n(\zeta) = \E_{\gamma^n}^\beta\Big[\prod_\mu \exp(\Re g(z_\mu)+\zeta \Im g(z_\mu) )\Big], \]
$n\ge 1$, which are bounded by
\begin{align*}
   &|f_n(\zeta)| \leq \E_{\gamma^n}^\beta[\prod_\mu \exp(\Re g(z_\mu)+\Re(\zeta) \Im g(z_\mu) )] \\
   &\leq \E_{\gamma^n}^\beta[\prod_\mu \exp(\Re g(z_\mu)+2 \Im g(z_\mu) )]+\E_{\gamma^n}^\beta[\prod_\mu \exp(\Re g(z_\mu)-2 \Im g(z_\mu) )]
\end{align*}
if $|\zeta|\leq 2$. Let
\begin{equation*}
    \mathbf{v}_\pm = \frac{1}{2}\begin{pmatrix}
( \sqrt{k} (\Re a_k\pm 2\Im a_k))_{k\geq1} \\
( \sqrt{k} (\Re b_k\pm 2\Im b_k))_{k\geq1}.
\end{pmatrix},
\end{equation*}
Since we assume that the theorem is true if $g$ is real, we see that if $n$ is large enough
\begin{align*}
   &\E_{\gamma^n}^\beta[\prod_\mu \exp(\Re g(z_\mu)\pm 2 \Im g(z_\mu) )] \\
   &\leq 2\exp \Big(\tfrac{2}{\beta}  \mathbf{v}_\pm^t (I+K)^{-1}\mathbf{v}_\pm+2\big(1-\tfrac{2}{\beta}\big) (\dd^t(I+K)^{-1}\mathbf{v}_\pm) \Big).
\end{align*}
We have that $\|K\|<1$ by the strengthened Grunsky's inequality, $\|\dd\|_2$ is bounded because of \eqref{eq:coeffdecay}, and 
\[ \|\mathbf{v}_\pm\|_2^2 = 4\sum_{k\geq 1} k(|a_k|^2+|b_k|^2)< \infty. \]
by our assumption on the regularity of $g$. Thus,
\[
 |f_n(\zeta)| \leq 2\sum_{\star=\pm} \exp \Big( \|\mathbf{v}_\star\|_2 (1-\|K\|)^{-1} \Big(\tfrac{2}{\beta} \|\mathbf{v}_\star\|_2 +2\big|1-\tfrac{2}{\beta}\big| \|\mathbf{d}\|_2 \Big)
 \]
uniformly on $|\zeta|\leq 2$, for all $n$ large enough. By Montel's theorem, the family $\{f_n\}_{n\geq 1}$ is normal on $|\zeta|<2$ so there is a subsequence converging uniformly on compact subsets. But the sequence itself converges pointwise on the real line, whence uniformly on $|\zeta|\leq 1$. In particular, it converges at $z=\i$ to the desired limit.
\end{proof}

\subsection{The integral equation}

From now on we will assume that the function $g$ on $\gamma$ is real-valued. The starting point for the analysis of the asymptotics of $D_n^\beta[e^g]$ is to 
make the change of variables $z_{\mu}= \phi(e^{\i \theta_{\mu}})$ in the integral in the right side of \eqref{eq:maindef} so that we get a particle system on the circle instead. 
We write
\[
\frac{\beta}{2}\sum_{1\leq \mu\neq \nu \leq n} \log \lvert  \phi(e^{\i \theta_{\mu}})-\phi(e^{\i \theta_{\nu}})\rvert + \sum_\mu\log|\phi'(e^{\i \theta_\mu})|=
\frac{\beta}{2}F_n(\theta) +\big(1-\frac{\beta}{2}\big)\sum_\mu\log|\phi'(e^{\i \theta_\mu})|,
\]
where
\begin{align}\label{Fn}
    F_n(\theta) &= \sum_{1\leq \mu,\nu \leq n} \log \Big\lvert \frac {\phi(e^{\i \theta_{\mu}})-\phi(e^{\i \theta_{\nu}})}{e^{\i \theta_{\mu}}-e^{\i \theta_{\nu}}} \Big\rvert + \sum_{1\leq \mu\neq\nu \leq n} \log\lvert e^{\i \theta_{\mu}}-e^{\i \theta_{\nu}}\rvert.
\end{align}
Note that when $\mu=\nu$, the term in the first sum on the right side of \eqref{Fn} equals $\sum_\mu\log|\phi'(e^{\i\theta_\mu})|$.
Let $\E_n^\beta$ denote expectation with respect to the measure on $H_n= \{ 0\leq \theta_1 < \theta_2 < \cdots < \theta_n \leq 2\pi \}$ with density  
$$
\frac 1{Z_{n,\beta}}\exp\big(\frac{\beta}{2}F_n(\theta) + \big(1-\frac{\beta}{2}\big)\sum_\mu\log|\phi'(e^{\i \theta_\mu})|\big). 
$$
With this notation,
\begin{align}\label{eq1}
    &\frac{D_n^\beta[e^g]}{D_n^\beta[1]} = \E_n^\beta\Big[\exp \Big( \sum_\mu g\circ\phi(e^{\i\theta_\mu}) \Big)\Big] \nonumber \\
    &= \frac{1}{D_n^\beta[1]} \int_{H_n} \exp \Big( \frac{\beta}{2} F_n(\theta)+  \sum_\mu g\circ\phi(e^{\i\theta_\mu}) +(1-\tfrac{\beta}{2})\log|\phi'(e^{\i \theta_\mu})| \Big) \d^n \theta \nonumber \\
    &= \frac{1}{D_n^\beta[1]n!} \int_{[0,2\pi]^n} \exp \Big( \frac{\beta}{2} F_n(\theta)+  \sum_\mu g\circ\phi(e^{\i\theta_\mu}) +(1-\tfrac{\beta}{2})\log|\phi'(e^{\i \theta_\mu})| \Big) \d^n \theta.
\end{align}
To analyze the asymptotics of this expression we will make a change of variables in \eqref{eq1}: we replace $\theta_\mu$ by $\theta_\mu-\frac{1}{n}h\circ \phi(e^{\-i\theta_\mu})$ where $h$ 
is a function on $\gamma$ that has to be chosen appropriately. In the case when $\gamma$ is the unit circle and $\beta=2$ it was seen in \cite{Joh88} that the right choice is to let $h$ be the conjugate function of $g$ on the unit circle. 
The \textit{conjugate function} $\Tilde{f}$ of a function $f(e^{\i\theta}) =f(\theta)= \sum_{k=1}^\infty \alpha_k\cos k\theta+\beta_k\sin k\theta$ is defined as 
\begin{align}\label{Conjfcn}
\Tilde{f}(e^{\i\omega}) =\Tilde{f}(\omega)= \mathrm{p.v. }\int_0^{2\pi} \cot\big(\frac{\omega-\theta}{2}\big) f(\theta) \frac{\d\theta}{2\pi} 
=\sum_{k=1}^\infty -\beta_k\cos k\omega+\alpha_k\sin k\omega,
\end{align}
where p.v. denotes the principal value.
In  the case of a general Jordan curve we need a generalization of the conjugate function.

It turns out, as will be seen in the proof of Theorem \ref{mainthm} in the next section, see Remark \ref{rem:inteq}, that $h$ should be chosen as the real-valued solution to the integral equation
\begin{equation}\label{inteqn}
g(z)=\Re\frac{\beta}{2\pi}\mathrm{p.v.}\int_\gamma \frac{h(\zeta)}{\zeta-z}d\zeta
\end{equation}
for $z \in\gamma$. In \eqref{inteqn} we can introduce the parametrization $z=\phi(e^{\i\theta})$, $0\le \theta\le 2\pi$, of $\gamma$, which gives 
the following integral equation on $\T$,
\begin{equation}\label{inteq}
g\circ\phi(e^{\i \omega}) = \Re p.v. \frac{\beta}{2} \int_0^{2\pi} \frac{\i   e^{\i \theta}\phi'(e^{\i \theta})}{\phi(e^{\i \theta})-\phi(e^{\i \omega})} h\circ \phi(e^{\i\theta}) \frac{\d\theta}{\pi}.  
\end{equation}
Note that if $\gamma=\T$ and $\beta=2$, then \eqref{inteq} gives $g=-\Tilde{h}$, i.e.~ $h=\Tilde{g}$.

To simplify the notation we set $G(\theta)= g\circ\phi(e^{\i\theta})$ and $H(\theta)=h\circ\phi(e^{\i\theta})$. We have the Fourier expansions
\[
G(\theta)=\sum_{k\ge 1}a_k\cos k\theta+b_k\sin k\theta, \quad H(\theta)=\sum_{k\ge 1}c_k\cos k\theta+d_k\sin k\theta,
\]
(recall that $\frac{a_0}{2} = \int_0^{2\pi} G(\theta) \d\theta = 0 $) and in analogy with \eqref{gvector}, we write
\begin{equation}\label{hvector}
 \h = \frac{1}{2} \begin{pmatrix}
    (\sqrt{k}c_k)_{k\geq 1} \\ (\sqrt{k}d_k)_{k\geq 1}
\end{pmatrix},
\end{equation}
so that
\[
G( \theta) = 2\begin{pmatrix}
    \mathbf{x_\theta} \\ \mathbf{y_\theta}
\end{pmatrix}^t\g,\quad
H( \theta) = 2\begin{pmatrix}
    \mathbf{x_\theta} \\ \mathbf{y_\theta}
\end{pmatrix}^t\h.
\]
The equation \eqref{inteq} becomes
\begin{equation}\label{GH}
G(\omega)=\frac{\beta}{2\pi}\Re\int_0^{2\pi}\frac{\i e^{\i\theta}\phi'(e^{\i\theta})}{\phi(e^{\i\theta})-\phi(e^{\i\omega})}H(\theta)\,d\theta.
\end{equation}
Also, we define
\[
L = \begin{pmatrix}
    0 & -I \\ 
     I & 0
\end{pmatrix},
\quad
J = \begin{pmatrix}
    (k\delta_{kj})_{k,j\geq1} & (0)_{k,j\geq 1} \\ 
     (0)_{k,j\geq 1} & (k\delta_{kj})_{k,j\geq1}
\end{pmatrix}
\]
as operators on $\ell_2(\R) \oplus \ell_2(\R)$. Note that we have the formulas
\begin{equation}\label{xyorth}
\int_0^{2\pi} \begin{pmatrix}
    \mathbf{x_\theta} \\ \mathbf{y_\theta}
\end{pmatrix} \begin{pmatrix}
    \mathbf{x_\theta} \\ \mathbf{y_\theta}
\end{pmatrix}^t \frac{\d\theta}{\pi} = J^{-1},
\end{equation}
and 
\begin{equation}\label{Hprime}
H'(\theta) = -2\begin{pmatrix}
    \mathbf{x_\theta} \\ \mathbf{y_\theta}
\end{pmatrix}^t JL \mathbf{h}, \quad 
\Tilde{H}( \theta) = 2\begin{pmatrix}
    \mathbf{x_\theta} \\ \mathbf{y_\theta}
\end{pmatrix}^t L\h.
\end{equation}

Our assumption that $g$ and hence $G$ is $C^{4+\alpha}$ will be quantified via the norm
\[
\|G\|_{4,\alpha}=\sup_{0\le\theta_1,\theta_2\le 2\pi}\frac{|G^{(4)}(\theta_1)-G^{(4)}(\theta_2)|}{|e^{\i\theta_1}-e^{\i\theta_2}|^\alpha}+\sum_{j=0}^4\|G^{(j)}\|_\infty.
\]
 Note that if $g\in C^{4+\alpha}(\gamma)$ then $G(\theta)=g\circ \phi(e^{\i\theta})$ is $C^{4+\alpha}$, since $\phi$ is $C^{9+\alpha}$, so
$\|G\|_{4,\alpha}<\infty$.
The solution to the integral equation is now provided by the following lemma.
\begin{lemma}\label{lem:inteqnsol}
Assume that $g$ is $C^{4+\alpha}$, $\alpha>0$, on the $C^{9+\alpha}$ Jordan curve $\gamma$, and let
\begin{equation}\label{hgeqn}
\h=\frac 2{\beta}L(I+K)^{-1}\g.
\end{equation}
Then the function
\[
H( \theta) = 2\begin{pmatrix}
    \mathbf{x_\theta} \\ \mathbf{y_\theta}
\end{pmatrix}^t\h,
\]
satisfies
\begin{equation}\label{HLip}
\|H\|_{4,\alpha}\le C\big(A(1-\|K\|)^{-1}+1\big)\|G\|_{4,\alpha},
\end{equation}
where $A$ is the constant in \eqref{eq:coeffdecay}.  Furthermore,
$G$ and $H$ satisfy the integral equation \eqref{GH}.
\end{lemma}

\begin{proof}
We can write \eqref{hgeqn} as
\[
\h=\frac 2{\beta}L\g-\frac 2{\beta}LK(I+K)^{-1}\g,
\]
so
\begin{equation}\label{HGeqn}
H(\theta)=\frac 2{\beta}\Tilde{G}(\theta)-\frac 4{\beta}\begin{pmatrix}
    \mathbf{x_\theta} \\ \mathbf{y_\theta}
\end{pmatrix}^tLK(I+K)^{-1}\g.
\end{equation}
It follows from Privalov's theorem that $\|\tilde{G}\|_{4,\alpha}\le C\|G\|_{4,\alpha}$. Also, if we write
\[
\begin{pmatrix} (\xi_k)_{k\ge 1} \\ (\eta_k)_{k\ge 1} \end{pmatrix} =LK(I+K)^{-1}\g,
\]
then using \eqref{eq:coeffdecay}, and the inequalities $((I+K)^{-1})_{jk}\le (1-\|K\|)^{-1}$ and $|a_l|,|b_l|\le C\|G\|_{4,\alpha}l^{-4-\alpha}$, we see that
\[
|\xi_k|+|\eta_k|\le CA(1-\|K\|)^{-1}\|G\|_{4,\alpha}k^{-6-\epsilon}.
\]
Hence, the $\|\cdot\|_{4,\alpha}$-norm of the function
\[
\frac 4{\beta}\begin{pmatrix}
    \mathbf{x_\theta} \\ \mathbf{y_\theta}
\end{pmatrix}^tLK(I+K)^{-1}\g
\]
is bounded by $CA(1-\|K\|)^{-1}\|G\|_{4,\alpha}$, and we have proved the estimate \eqref{HLip}. We see that \eqref{GH} can be written
\begin{align*}
G(\omega) &= \frac{\beta}{2} \Re \mathrm{p.v.} \int_0^{2\pi} \frac{\partial}{\partial \theta}\log \big( \phi(e^{\i\theta})-\phi(e^{\i \omega}) \big) H(\theta) \frac{\d\theta}{\pi} \\
&=\frac{\beta}{2} \Re \mathrm{p.v.} \int_0^{2\pi} \frac{\partial}{\partial \theta}\bigg(\log \frac {\phi(e^{\i\theta})-\phi(e^{\i \omega})}{e^{\i\theta}-e^{\i\omega}}+\log(e^{\i\theta}-e^{\i\omega})\bigg)
H(\theta) \frac{\d\theta}{\pi} \\
&=\frac{\beta}{2} \Re \mathrm{p.v.} \int_0^{2\pi} \frac{\i e^{\i\theta}}{e^{\i\theta}-e^{\i\omega}}H(\theta)\frac{\d\theta}{\pi}-\frac{\beta}{2}\int_0^{2\pi}\log\left|
\frac {\phi(e^{\i\theta})-\phi(e^{\i \omega})}{e^{\i\theta}-e^{\i\omega}}\right|H'(\theta)\frac{\d\theta}{\pi}\\
&=-\frac{\beta}{4\pi} \mathrm{p.v.} \int_0^{2\pi} \cot\frac{\omega-\theta}2 H(\theta)\,d\theta-\frac{\beta}{2}\int_0^{2\pi}\log\left|
\frac {\phi(e^{\i\theta})-\phi(e^{\i \omega})}{e^{\i\theta}-e^{\i\omega}}\right|H'(\theta)\frac{\d\theta}{\pi}.
\end{align*}
If we use \eqref{KGrunskyrel}, \eqref{xyorth}  and \eqref{Hprime}, we obtain
\begin{equation}\label{GH2}
G(\omega)=-\frac{\beta}2\Tilde{H}(\omega)-\frac{\beta}{\pi}\int_0^{2\pi}\begin{pmatrix} \mathbf{x_\omega} \\ \mathbf{y_\omega}
\end{pmatrix}^tK\begin{pmatrix} \mathbf{x_\theta} \\ \mathbf{y_\theta}
\end{pmatrix}\begin{pmatrix}\mathbf{x_\theta} \\ \mathbf{y_\theta}
\end{pmatrix}^tJL\h\,d\theta=-\frac{\beta}2\Tilde{H}(\omega)-\beta\begin{pmatrix} \mathbf{x_\omega} \\ \mathbf{y_\omega}
\end{pmatrix}^tKL\h.
\end{equation}
We see that \eqref{GH2} can be written as
\begin{equation}\label{inteq2}
    -\mathbf{g} = \frac{\beta}{2}(I+K)L\mathbf{h}.
\end{equation}
in Fourier form. If $\h$ is defined by \eqref{hgeqn} then \eqref{inteq2} holds. Working backwards in the argument above, we see that $G$ and $H$ satisfy \eqref{GH}.
\end{proof}

\begin{remark}
Since we are only working with very regular functions and curves in this paper we will not discuss the integral equation \eqref{inteqn} under weaker regularity conditions. Note that 
the equation \eqref{inteq2} on the Fourier side is meaningful if $G$ and $H$ are $H^{1/2}$ functions on $\T$. See Section 2 in \cite{JoVi}.
\end{remark}

\begin{remark}
If we look at \eqref{HGeqn}, we see that $H$ is the conjugate function of $\frac 2{\beta}(G-V)$, where
\[
V(\theta)=2\begin{pmatrix}
    \mathbf{x_\theta} \\ \mathbf{y_\theta}
\end{pmatrix}^tK(I+K)^{-1}\g.
\]
If we write
\[
V(\theta)=\sum_{k=1}^\infty(\bar{v}_ke^{\i k\theta}+v_ke^{-\i k\theta}),\quad G(\theta)=\sum_{k=-\infty}^\infty g_k e^{\i k\theta}
\]
in complex Fourier form, a computation shows that \eqref{inteq2} is equivalent to the system of equations
\[
\sum_{k=1}^\infty ka_{k\ell}\bar{v}_k+v_\ell=\sum_{k=1}^\infty ka_{k\ell}g_k,
\]
which was used in \cite{Joh88}. This type of system of equations goes back to \cite{Pom69} and the study of the location of the Fekete points on $\gamma$.
\end{remark}

In the proof of Theorem \ref{mainthm} we will need the formula in the next lemma.

\begin{lemma}\label{lem:gvar} The following identity holds
\begin{align}\label{eq:I+K}
    &\frac{1}{4\pi} \int_0^{2\pi} G(\theta)H'(\theta) \d \theta + 2(1-\tfrac{\beta}{2})\frac{1}{4\pi} \int_0^{2\pi} \log |\phi'(e^{\i \theta})|H'(\theta) \d \theta \nonumber \\
&= \frac{2}{\beta}  \g^t (I+K)^{-1}\g+2\big(1-\frac{2}{\beta}\big) \dd^t(I+K)^{-1}\g.
\end{align}
\end{lemma}

\begin{proof}
It follows from \eqref{Hprime} and \eqref{hgeqn} that
\[
H'(\theta)=-2\begin{pmatrix}
    \mathbf{x_\theta} \\ \mathbf{y_\theta}
\end{pmatrix}^tJL\h=\frac 4{\beta}\begin{pmatrix}
    \mathbf{x_\theta} \\ \mathbf{y_\theta}
\end{pmatrix}^tJ(I+K)^{-1}\g,
\]
since $L^2=-I$. Thus, using \eqref{logphiprime} and \eqref{xyorth}, we obtain
\begin{align*}
&\frac{1}{4\pi} \int_0^{2\pi} G(\theta)H'(\theta) \d \theta + 2(1-\tfrac{\beta}{2})\frac{1}{4\pi} \int_0^{2\pi} \log |\phi'(e^{\i \theta})|H'(\theta) \d \theta\\
&=\frac{2}{\beta} \int_0^{2\pi} \Big( \g^t -2\big(1-\frac{\beta}{2}\big)\dd^t\Big)\begin{pmatrix}
    \mathbf{x_\theta} \\ \mathbf{y_\theta}
\end{pmatrix} \begin{pmatrix}
    \mathbf{x_\theta} \\ \mathbf{y_\theta}
\end{pmatrix}^t J(I+K)^{-1} \mathbf{g} \frac{\d\theta}{\pi}\\
&= \frac{2}{\beta}  \g^t (I+K)^{-1}\g+2\big(1-\frac{2}{\beta}\big) \dd^t(I+K)^{-1}\g.
\end{align*}

\end{proof}

Recall \eqref{gbetavec} and let
\begin{equation}\label{hbetavec}
\h_\beta=\frac 2{\beta}L(I+K)^{-1}\g_\beta.
\end{equation}
Define $h_\beta$ on $\gamma$ by
\begin{equation}\label{hbeta}
h_{\beta}\circ\phi(e^{\i\theta})=2\begin{pmatrix}
    \mathbf{x_\theta} \\ \mathbf{y_\theta}
\end{pmatrix}^t\h_\beta.
\end{equation}
It follows from Lemma \ref{lem:inteqnsol} that $g_\beta$ and $h_\beta$ satisfy
\begin{equation}\label{betainteqn}
g_\beta(z)=\Re\frac{\beta}{2\pi}\mathrm{p.v.}\int_\gamma \frac{h_\beta(\zeta)}{\zeta-z}d\zeta,
\end{equation}
for $z\in\gamma$. The integral equations \eqref{inteqn} and \eqref{betainteqn} can be solved without moving the problem to the unit circle and using the Grunsky operator. Let $\nu=\nu(z)$,
$z\in\gamma$, denote the exterior normal to $\gamma$. For $w\in\C\setminus\gamma=\Omega_+\cup\Omega_-$, the \emph{single-layer potential} with real-valued density $f$ on $\gamma$ is defined by
\[
S(f)(w)=\frac 1{2\pi}\int_\gamma\log|\zeta-w|^{-1}f(\zeta)\,|d\zeta|.
\]
This is a harmonic function in $\Omega_+\cup\Omega_-$. The single-layer potential is continuous across $\gamma$ but its normal derivatives have a jump. 
Let $\frac{\partial S_\pm(f)}{\partial\nu}(z)$, $z\in\gamma$, denote the derivatives in the direction $\pm \nu$. Then,
\begin{equation}\label{Plemelj}
\frac{\partial S_\pm(f)}{\partial\nu}(z)=\frac 1{2\pi}\int_\gamma\frac{\partial }{\partial\nu(z)}\log|\zeta-z|^{-1}f(\zeta)\,|d\zeta|\mp\frac 12 f(z),
\end{equation}
see e.g.~ \cite{Kress}.
Let $\frac{\partial}{\partial s}$ denote the tangential derivative along the curve $\gamma$. If $\zeta=\zeta(t)$, $a\le t\le b$, is some parametrization of $\gamma$, then
\begin{align*}
&\Re\frac{\beta}{2\pi}\mathrm{p.v.}\int_\gamma\frac{h(\zeta)}{\zeta-z}d\zeta=\frac{\beta}{2\pi}\mathrm{p.v.}\int_a^b\Big(\Re\frac{\zeta'(t)}{\zeta(t)-z}\Big)h(\zeta(t))\,dt\\
&=\frac{\beta}{2\pi}\mathrm{p.v.}\int_a^b\frac{\partial}{\partial t}\log|\zeta(t)-z|h(\zeta(t))\,dt=\frac{\beta}{2\pi}\int_a^b\log|\zeta(t)-z|^{-1}\frac{\partial}{\partial t}h(\zeta(t))\,dt\\
&=\frac{\beta}{2\pi}\int_\gamma\log|\zeta-z|^{-1}\frac{\partial h}{\partial s}(\zeta)\,|d\zeta|=\beta S\big(\frac{\partial h}{\partial s}\big)(z),
\end{align*}
where we used $\frac{\partial h}{\partial s}(\zeta(t))=|\zeta'(t)|^{-1}\frac{\partial}{\partial t}h(\zeta(t))$.
Since the single-layer potential is continuous across $\gamma$, we see that if $g$ and $h$ satisfy \eqref{inteqn}, then
\[
g_\pm(w):=\beta S\big(\frac{\partial h}{\partial s}\big)(w),\quad w\in\Omega_\pm,
\]
are harmonic extensions of $g$ to $\Omega_\pm$. It follows from \eqref{Plemelj} that
\begin{equation}\label{dhs}
\frac{\partial h}{\partial s}(z)=-\frac 1{\beta}\Big(\frac{\partial g_+}{\partial \nu}(z)-\frac{\partial g_-}{\partial \nu}(z)\Big).
\end{equation}
Let $\Tilde{g}_\pm$ be the conjugate harmonic function to $g_\pm$ in $\Omega_\pm$. Then, by the Cauchy-Riemann equations
\begin{equation*}
\frac{\partial h}{\partial s}(z)=-\frac 1{\beta}\Big(\frac{\partial \Tilde{g}_+}{\partial s}(z)-\frac{\partial \Tilde{g}_-}{\partial s}(z)\Big), 
\end{equation*}
and integrating this gives
\[
h(z)=-\frac 1{\beta}(\Tilde{g}_+(z)-\Tilde{g}_-(z)).
\]
This gives another solution formula for the integral equation \eqref{inteqn}.

We are now in position to prove Proposition \ref{prop:diffvar}.
\begin {proof}[ Proof of Proposition \ref{prop:diffvar}]
We see from \eqref{hbetavec}, \eqref{Hprime}, \eqref{xyorth}, \eqref{hbeta} and \eqref{dhs} that
\begin{align}\label{Neu}
\g^t(I+K)^{-1}\g&= -\frac{\beta}{2} \g^t L\h= -\frac {\beta}{2\pi}\int_0^{2\pi}\g^t \begin{pmatrix}
    \mathbf{x_\theta} \\ \mathbf{y_\theta}\end{pmatrix} \begin{pmatrix}
    \mathbf{x_\theta} \\ \mathbf{y_\theta}\end{pmatrix}^tJL\h_\beta\,d\theta\notag\\
 &=-\frac \beta{8\pi}\int_0^{2\pi}g(\phi(e^{\i\theta}))\frac{\partial}{\partial\theta}h(\phi(e^{\i\theta}))\,d\theta=
 -\frac \beta{8\pi}\int_\gamma g(\zeta)\frac{\partial h}{\partial s}(\zeta)\,|d\zeta|\notag\\
 &= \frac 1{8\pi}\int_\gamma g(\zeta)\Big(\frac{\partial g_{+}}{\partial \nu}(\zeta)-\frac{\partial g_{-}}{\partial \nu}(\zeta)\Big)\,|d\zeta|.
 \end{align}
 Since $g_{\pm}$ are harmonic in $\Omega_{\pm}$ it follows from the first Green's theorem that this equals the right side of \eqref{Dirichlet}. It can be checked that if $G$ is $C^{1+\alpha}$ and $\gamma$ is $C^{4+\alpha}$ then Lemma \ref{lem:inteqnsol} and its proof still holds and gives an $H$ that is $C^{1+\alpha}$.

\end{proof}

The \textit{Neumann-Poincar\'e operator} $\mathcal{K}_{\mathrm{NP}}$ on a $C^{2+\alpha}$ curve $\gamma$ is defined by
\[
\mathcal{K}_{\mathrm{NP}}(g)(z)=\frac1{\pi} p.v. \int_\gamma g(\zeta)\frac{\partial}{\partial \nu(\zeta)}\log|z-\zeta|\,|d\zeta|, \quad z \in \gamma
\]
where $g$ is a real-valued function on $\gamma$. It is also called the double-layer potential operator. The eigenvalues of $\mathcal{K}_{\mathrm{NP}}$ acting on the Sobolev space $H^{1/2}$ are called the Fredholm eigenvalues of $\gamma$ and have been much studied, see e.g.~ \cite{DPS}, \cite{Sch} and references therein. By taking the real and imaginary parts of the top-right equation in (70) in \cite{Sch}, we see that the eigenvalues of the operator $K$ in \eqref{eq:K} are exactly the Fredholm eigenvalues (note that in \cite{Sch}, the author defines the Fredholm eigenvalues to be their reciprocals). Hence $\det(I+K)=\det(I+\mathcal{K}_{\mathrm{NP}})$. Moreover, the eigenvalues of $K$ correspond to the singular values of the Grunsky operator $B$ and their negatives (see Lemma 2.1 in \cite{Joh22}), so $\det(I-BB^*)=\det(I-|B|)\det(I+|B|)=\det(I+K)= \det(I+\mathcal{K}_{\mathrm{NP}})$. The next proposition gives a more direct relation between $K$ and 
$\mathcal{K}_{\mathrm{NP}}$.

\begin{prop}\label{prop:NP}
Assume that $\gamma$ is $C^{4+\alpha}$. Let $g$ be a real-valued $C^{1+\alpha}$ function on $\gamma$ with mean zero with respect to the equilibrium measure on $\gamma$
and Fourier expansion
\[
g\circ\phi(e^{\i\theta})=2\begin{pmatrix}
    \mathbf{x_\theta} \\ \mathbf{y_\theta}\end{pmatrix}^t\g.
\]
Then,
\[
\mathcal{K}_{\mathrm{NP}}(g)\circ\phi(e^{\i\theta})=2\begin{pmatrix}
    \mathbf{x_\theta} \\ \mathbf{y_\theta}\end{pmatrix}^tK\g.
\]
\end{prop}

\begin{proof}
Let $s(\zeta)$ be the unit tangent vector at $\zeta\in\gamma$ as a complex number. Then
\[
\frac{\partial}{\partial\nu(\zeta)}\log|z-\zeta|=\Re\Big(\frac{-\i s(\zeta)}{\zeta-z}\Big),
\]
and consequently
\[
\mathcal{K}_{\mathrm{NP}}(g)(z)=\Re\Big(\frac1{\pi \i}\mathrm{p.v.}\int_\gamma \frac{s(\zeta)}{\zeta-z}g(\zeta)\,|d\zeta|\Big).
\]
We see that
\begin{align*}
\mathcal{K}_{\mathrm{NP}}(g)\circ\phi(e^{\i\omega})&=\Re\Big(\frac1{\pi \i}\mathrm{p.v.}\int_0^{2\pi}\frac{\i\phi'(e^{\i\theta})}{\phi(e^{\i\theta})-\phi(e^{\i\omega})}
g(\phi(e^{\i\theta}))\,d\theta\Big)\\
&=\Re\Big(\frac1{\pi \i}\int_0^{2\pi}\Big(\frac{\i\phi'(e^{\i\theta})}{\phi(e^{\i\theta})-\phi(e^{\i\omega})}
-\frac{\i e^{\i\theta}}{e^{\i\theta}-e^{\i\omega}}\Big)g(\phi(e^{\i\theta}))\,d\theta\Big)
\end{align*}
since
\[
\Re\Big(\frac{e^{\i\theta}}{e^{\i\theta}-e^{\i\omega}}\Big)=\frac 12,
\]
and $g(\phi(e^{\i\theta}))$ has mean zero on the unit circle.
It follows by differentiation of \eqref{Grunsky} that
\[
\frac{\i\phi'(e^{\i\theta})}{\phi(e^{\i\theta})-\phi(e^{\i\omega})}
-\frac{\i e^{\i\theta}}{e^{\i\theta}-e^{\i\omega}}=\sum_{k,\ell\ge 1}\i ka_{k\ell}e^{-\i k\theta-\i\ell\omega}.
\]
Recall that $K$ is defined by \eqref{eq:K} where $B=B^{(1)}+\i B^{(2)}$ and that $g\circ \phi$ is given in \eqref{ggvec}. Thus,
\begin{align*}
\mathcal{K}_{\mathrm{NP}}(g)\circ\phi(e^{\i\omega})&=\sum_{k,\ell\ge 1}k\Re\big[a_{k\ell}(a_k-\i b_k)(\cos\ell\omega-\i\sin\ell\omega)\big]\\
&=2\sum_{k,\ell\ge 1}\Re\big[(b_{k\ell}^{(1)}+\i b_{k\ell}^{(2)})(\frac 12\sqrt{k}a_k-\i \frac 12\sqrt{k}b_k)(\frac{\cos\ell\omega}{\sqrt{\ell}}-\i\frac{\sin\ell\omega}{\sqrt{\ell}})\big]=
2\begin{pmatrix}
    \mathbf{x_\omega} \\ \mathbf{y_\omega}\end{pmatrix}^tK\g
\end{align*}
\end{proof}
The regularity conditions on $\gamma$ and $g$ are such that the above computation works without difficulties. They can be weakened but we will not enter into this here.

\begin{remark} \label{rem:NJO}
The operator
\[
\mathcal{N}g(z)=\frac{\partial g_+}{\partial \nu}(z)-\frac{\partial g_-}{\partial \nu}(z),
\]
where $z\in\gamma$, is called the \textit{Neumann jump operator}. From \eqref{Neu} we see that
\[
\frac 2{\beta}\g_\beta^t(I+K)^{-1}\g_\beta=\frac 1{¢\pi\beta}\int_\gamma g_\beta\mathcal{N}g_\beta\, ds,
\]
where we integrate with respect to the arclength measure. This formula agrees with formula (4.23) in \cite{WieZab22}, (note that their $\beta$ is our $\beta/2$). If we also use
formula (D16) in \cite{WieZab22}, and asymptotics of the partition function $Z_{n,\beta}(\T)$ for the unit circle, \eqref{Selberg}, we see that the asymptotic formula in \cite{WieZab22}
agrees with that in Theorem \ref{mainthm2}.

\end{remark}

\section{A relative Szeg\H{o} theorem}\label{sec:relSz}
In the rest of the paper $C$ will denote a constant that is independent of the curve $\gamma$ and the function $g$ defined on the curve. It can depend on $\beta$ but is always independent 
of $n$. If the constant depends on say $\gamma$ this will be indicated by $C(\gamma)$, and if it also depends on some norm $\|G\|$ of $G=g\circ\phi$ we will denote the constant by $C(\gamma,\|G\|)$. The precise value of these constants can change between formulas.

\subsection{Deforming the curve $\gamma$}
In the proof of the asymptotic formula for the partition function $Z_{n,\beta}(\gamma)$, Theorem \ref{mainthm2}, we will use a deformation $\gamma_s$ of the curve $\gamma$. Let 
$\phi_s(z)=s\phi(z/s)$, where $\phi$ is the exterior mapping function, and define $\gamma_s=\phi_s(\T)$ for $s\in[0,1]$ Then $\phi_0(z) = z$, $\mathrm{cap}(\gamma_s)=\mathrm{cap}(\gamma)=1$, and
by \eqref{Grunsky}
\begin{equation}\label{Grunskys}
\log\frac{\phi_s(z)-\phi_s(w)}{z-w}=\log\frac{\phi(z/s)-\phi(w/s)}{z/s-w/s}=-\sum_{k,l\ge 1} s^{k+l}a_{kl}z^{-k}w^{-l}.
\end{equation}
Set $a_{kl}(s)=s^{l+l}a_{kl}$, $b_{kl}(s)=\sqrt{kl}a_{kl}(s)$, and write $B(s):=(b_{kl}(s))_{k,l\ge 1}=B^{(1)}(s)+\i B^{(2)}(s)$ for the Grunsky operator for $\gamma_s$. Note that $|a_{kl}(s)|\le |a_{kl}|$ for $s\in[0,1]$.
As in \eqref{eq:K} and \eqref{eq:d}, we define
\[ K(s) = \begin{pmatrix}
B^{(1)}(s) & B^{(2)}(s) \\
B^{(2)}(s) & -B^{(1)}(s)
\end{pmatrix}, \quad \dd(s) = \frac{1}{2}\begin{pmatrix}
( \sqrt{k}s^k \sum_{j=1}^{k-1} \Re a_{j,k-j} )_{k\geq1} \\
( \sqrt{k}s^k \sum_{j=1}^{k-1} \Im a_{j,k-j} )_{k\geq1}
\end{pmatrix}. \]
The strengthened Grunsky inequality gives
\begin{align*}
\sum_{k\ge 1}|b_{kl}(s)w_l|^2&=\sum_{k\ge 1}s^{2k}\bigg|\sum_{l\ge 1}b_{kl}s^lw_l\bigg|^2\le s^2\sum_{k\ge 1}\bigg|\sum_{l\ge 1}b_{kl}s^lw_l\bigg|^2\\
&\le s^2\kappa^2\sum_{k\ge 1}s^{2k}|w_k|^2\le\kappa^2s^4\sum_{k\ge 1}|w_k|^2,
\end{align*}
for any complex sequence $(w_k)_{k\ge 1}$ in $\ell^2(\N)$, where $\kappa<1$. Thus, $\|B(s)\|\le \kappa s^2$ and hence $\|K(s)\|\le \kappa s^2\le \kappa$, for all $s\in [0,1]$.

Write
\begin{align}\label{Fns}
    F_n^s(\theta) &= \sum_{1\leq \mu,\nu \leq n} \log \big\lvert \frac {\phi_s(e^{\i \theta_{\mu}})-\phi_s(e^{\i \theta_{\nu}})}{e^{\i \theta_{\mu}}-e^{\i \theta_{\nu}}} \big\rvert + \sum_{1\leq \mu\neq\nu \leq n} \log\lvert e^{\i \theta_{\mu}}-e^{\i \theta_{\nu}}\rvert,
\end{align}
and let $\E_n^{\beta,s}$ denote expectation with respect to the measure on $H_n= \{ 0\leq \theta_1 < \theta_2 < \cdots < \theta_n \leq 2\pi \}$ with density proportional to
\[
\exp\Bigg(\frac{\beta}2 F_n^s(\theta)+\Big(1-\frac{\beta}2\Big)\sum_\mu\log|\phi_s'(e^{\i\theta_\mu})|\Bigg),
\]
so that $E_n^\beta=\E_n^{\beta,1}$, and $F_n(\theta)=F_n^1(\theta)$. Also, we let
\begin{equation}\label{Dns}
D_n^{\beta,s}[e^{g_s}]=\frac 1{n!}\int_{[0,2\pi]^n}\exp\Bigg[\frac{\beta}2 F_n^s(\theta)+\Big(1-\frac{\beta}2\Big)\sum_\mu\log|\phi_s'(e^{\i\theta_\mu})|+\sum_\mu G_s(\theta_\mu)\Bigg]
d^n\theta,
\end{equation}
where $G_s(\theta)=g_s\circ\phi_s(e^{\i\theta})$ and $g_s$ is a given real-valued function on $\gamma_s$. Let $\g(s)$ be the vector associated with $g_s$ as in \eqref{gvector}.

Theorem \ref{mainthm} follows from the following stronger version, which we will need in 
the proof of the asymptotics of $Z_{n,\beta}(\gamma)$.
\begin{thm}\label{smainthm}
Assume that $\gamma$ is a $C^{9+\alpha}$ Jordan curve, $g_s\in C^{4+\alpha}(\gamma_s)$, for some $\alpha>0$, and $\int_{\gamma_s}g_s\,d\nu_{\mathrm{eq}}^s=0$,
where $\nu_{\mathrm{eq}}^s $ is the equilibrium measure on $\gamma_s$. 
There is a constant $C(\gamma,\|G_s\|_{4,\alpha})$ such that
\begin{equation}\label{expbound}
\E_n^{\beta,s}\big[e^{\sum_\mu G_s(\theta_\mu)}\big]\le C(\gamma,\|G_s\|_{4,\alpha}),
\end{equation}
for $n\ge 1$, $s\in[0,1]$. Also, for each fixed $s\in[0,1]$,
\begin{equation}\label{eq:smain}
   \lim_{n\to\infty}\E_{n}^{\beta,s} \big[ e^{\sum_{\mu=1}^n G_s(\theta_\mu)} \big] = \exp \Big(\frac{2}{\beta}  \g(s)^t (I+K(s))^{-1}\g(s)+2\big(1-\frac{2}{\beta}\big) \dd(s)^t(I+K(s))^{-1}\g(s)\Big).
\end{equation}
\end{thm}

\subsection{A smaller domain of integration}

The first step of the proof of Theorem \ref{smainthm} is to show that we can restrict the domain of integration in \eqref{Dns} to a domain where $F_n^s$ is close to its maximum, without changing the asymptotics. The maximum of
\[ \sum_{1\leq \mu <\nu \leq n} \log \big| z_\mu-z_\nu \big|,\quad z_\mu\in\gamma, \]
is attained, by definition, at the Fekete points. These are close to the images under $\phi$ of the Fekete points of the unit circle, which are any rotation of the $n$th roots of unity, see \cite{Pom69}. 
Given $\theta\in H_n$, we define $t_\mu=t_\mu(\theta)$ by
\begin{equation}\label{tmu}
\theta_\mu=\frac{2\pi\mu}n+\sigma_n(\theta)+t_\mu(\theta),
\end{equation}
where $\sigma_n(\theta)$ is chosen so that $\sum_\mu t_\mu(\theta)=0$.
Set $\alpha_\mu = 2\pi\mu/n+\sigma_n$. Then the maximum of $F_n(\theta)$ should be close to
\begin{align*}
    &\sum_{1\leq \mu,\nu \leq n} \log \big\lvert \frac {\phi(e^{\i \alpha_{\mu}})-\phi(e^{\i \alpha_{\nu}})}{e^{\i \alpha_{\mu}}-e^{\i \alpha_{\nu}}} \big\rvert + \sum_{\mu\neq \nu} \log |2\sin{\pi(\mu-\nu)/n}| \\
    &= -\mathrm{Re} \sum_{k,l\geq 1} a_{kl} \sum_{\mu=1}^n e^{-i k\alpha_\mu} \sum_{\nu=1}^n e^{-i l\alpha_\nu} + \sum_{\mu\neq \nu} \log |2\sin{\pi(\mu-\nu)/n}|
\end{align*}
Now observe that $\sum_{\mu=1}^n e^{-\i k\alpha_\mu} = 0$ unless $k$ is a multiple of $n$ and
\begin{align*}
    \prod_{\mu\neq \nu} |2\sin{\pi(\mu-\nu)/n}| = \prod_{1\leq \mu<\nu\leq n} |e^{2\pi\i \mu/n}-e^{2\pi\i \nu/n}|^2 = \prod_{\mu=1}^{n-1} |1-e^{2\pi\i \mu/n}|^n = n^n
\end{align*}
since
\[ \prod_{\mu=1}^{n-1} (z-e^{2\pi\i \mu/n}) = \frac{z^n-1}{z-1},\]
which approaches $n$ as $z\to 1$. Recall that the assumption that $\gamma$ is a $C^{9+\alpha}$ curve means that the estimate \eqref{eq:coeffdecay} holds with $m=9$, and we will use it several times below. This estimate now yields the lower bound
\[ F_n^s(\alpha) \geq - n^2 \sum_{k,l\geq 1} |a_{nk, l}| + n\log n \geq -CA + n\log n.\]
We can now define the smaller domain of integration and estimate the probability of its complement.
\begin{lemma}\label{lemma:En}
Assume that $X_n$ is a bounded function on $[0,2\pi]^n$. Given a constant $K$, we define 
\begin{equation}\label{eq:En}
    E_ {n,s} \coloneqq \{ \theta\in H_n\ | \ F_n^s(\theta) > n\log n -Kn\}.
\end{equation}
Then, there is a constant $C(\gamma)$ such that
\begin{equation}\label{Enest}
\E_n^{\beta,s} [e^{X_n}\mathbbm{1}_{E_{n,s}^c}]\le e^{(C(\gamma)-K)n+\|X_n\|_\infty}.
\end{equation}
\end{lemma}

\begin{proof}

We want to estimate $ D_n^{\beta,s}[1]$ from below. Since $\log|\phi_s'(e^{\i\theta})|\le \log\|\phi_s'\|_\infty$ on $[0,2\pi]$, there is a constant $C(\gamma)$ such that
\begin{align*}
    D_n^{\beta,s}[1] = \int_{H_n} \exp\Big( \frac{\beta}{2} F_n^s(\theta) + (1-\tfrac{\beta}{2})\sum_\mu\log|\phi_s'(e^{\i \theta_\mu})| \Big) \d^n \theta 
    \geq  e^{-C(\gamma)n} \int_{H_n} \exp\Big( \frac{\beta}{2} F_n^s(\theta) \Big) \d^n \theta.
\end{align*}
Now,
\[  
D_n^{\beta,s}[1] \geq e^{-C(\gamma)n} \int_{\sup_{\mu} |t_\mu(\theta)|\leq 1/n } \exp \Big(\frac{\beta}{2} F_n^s(\theta) \Big) \d^n \theta.
\]
Consider the first sum in \eqref{Fns}. By \eqref{Grunskys}, 
\[ 
\Big\lvert \sum_{1\leq \mu,\nu \leq n} \log \big\lvert \frac {\phi_s(e^{\i \theta_{\mu}})-\phi_s(e^{\i \theta_{\nu}})}{e^{\i \theta_\mu}-e^{\i \theta_\nu}} \big\rvert  \Big\rvert
    \leq \sum_{k,l \geq 1} |a_{kl}| \Big\lvert \sum_{\mu=1}^n e^{-\i k \theta_{\mu}} \Big\rvert \Big\lvert \sum_{\nu= 1}^n e^{-\i l\theta_{\nu}} \Big\rvert,
    \]
for $0\le s \le 1$.
If $k$ is not divisible by $n$, 
\[ \Big\lvert \sum_{\nu= 1}^n e^{-\i k\theta_{\nu}} \Big\rvert = \Big\lvert \sum_{\mu= 1}^n (e^{-\i k\theta_{\mu}} - e^{-\i k\alpha_\mu}) \Big\rvert \leq k \sum_{\mu=1}^n |t_{\mu}| \leq k \]
if $\sup_{\mu} |t_\mu| \leq 1/n$. If $k\ge 1$ is a multiple of $n$ then the sum is simply bounded by $n\leq k$. Thus, 
\begin{align}\label{eq:ens}
    \Big\lvert \sum_{1\leq \mu,\nu \leq n} \log \big\lvert \frac {\phi_s(e^{\i \theta_{\mu}})-\phi_s(e^{\i \theta_{\nu}})}{e^{\i \theta_\mu}-e^{\i \theta_\nu}} \big\rvert  \Big\rvert
    \leq \sum_{k,l\geq 1} kl |a_{kl}| \le C(\gamma)
\end{align}
because of the fast decay of the Grunsky coefficients, \eqref{eq:coeffdecay}. Now consider the second sum in \eqref{Fns}. Set  $\beta_\mu = \alpha_\mu+\tau t_\mu$, and $f(\tau) = \sum_{\mu\neq \nu} \log |e^{\i \beta_\mu}-e^{\i \beta_\nu}|$. Then $f(0)=n\log n$ and
\begin{align*}
    &f'(0) = \sum_{\mu\neq \nu} \cot\Big(\frac{\alpha_\mu-\alpha_\nu}{2}\Big)\frac{t_\mu-t_\nu}{2} = \sum_{k=1}^{n-1} \cot(\pi k/n)\sum_{\mu=1}^n t_\mu = 0.
\end{align*}
Thus,
\begin{equation}\label{eq:LogSinDiff}
     f(0)-f(1) = - \int_0^1 (1-\tau) f''(\tau) \d \tau = \int_0^1 (1-\tau) \sum_{\mu\neq\nu} \frac{(t_\mu-t_\nu)^2}{4\sin^2((\beta_\mu-\beta_\nu)/2)}\d\tau.   
\end{equation}

If $\max_{\mu} |t_\mu| \leq 1/n$ for $1\le\mu\le n$ and $\mu\neq\nu$, then, then
\[ 
\Big\lvert \sin\Big( \frac{\beta_\mu-\beta_\nu}{2} \Big)-\sin\Big( \frac{\alpha_\mu-\alpha_\nu}{2} \Big) \Big\rvert \leq \frac{|t_\mu-t_\nu|}{2} \leq \frac{1}{n} \leq \frac 12\Big\lvert\sin\frac{\alpha_\mu-\alpha_\nu}2\Big\rvert
\]
since
\[
\Big\lvert\sin\frac{\alpha_\mu-\alpha_\nu}2\Big\rvert\ge\sin\frac{\pi}n\ge \frac 2n,
\]
and thus
\[
4\sin^2\Big( \frac{\beta_\mu-\beta_\nu}{2} \Big) \geq \sin^2\Big( \frac{\alpha_\mu-\alpha_\nu}{2} \Big). 
\]
This yields
\[n\log n - \sum_{\mu\neq \nu} \log |e^{\i \theta_\mu}-e^{\i \theta_\nu}| \leq \sum_{\mu\neq\nu} \frac{(t_\mu-t_\nu)^2
}{\sin^2(\pi(\mu-\nu)/n)} \leq \frac{4}{n} \sum_{k=1}^{n-1} \sin^{-2}(\pi k/n) = \frac{4}{3}(n-1/n). \]
Therefore,
\begin{align*}
    D_n^{\beta,s}[1] \geq e^{-C(\gamma)n} \int_{\sup_{\mu} |t_\mu|\leq 1/n } \exp \Big(\frac{\beta}{2} F_n^s(\theta) \Big) \d^n \theta \geq e^{-C(\gamma)n}  n^{(\beta/2-1)n}.
\end{align*}
Combining this with the definition of $E_{n,s}$ and the bound on $\log|\phi_s'(e^{\i\theta})|$, we obtain the estimate
\begin{align*}
    \E_n^{\beta,s} [e^{X_n}1_{E_{n,s}^c}] &=\frac{1}{D_n^{\beta,s}[1]}\int_{E_{n,s}^c} \exp(\frac{\beta}{2}F_n^s(\theta) +(1-\tfrac{\beta}{2})\sum_\mu\log|\phi_s'(e^{\i \theta_\mu})| +X_n(\theta)) \d^n\theta  \\
    &\leq e^{C(\gamma)n}n^{-(\beta/2-1)n} \int_{E_{n,s}^c} \exp (\frac{\beta}{2}(n\log n -Kn)+\|X_n\|_\infty) \d^n \theta \\
    &\leq e^{(C(\gamma)-K)n}n^n\frac{(2\pi)^n}{n!}e^{\|X_n\|_\infty}\le e^{(C(\gamma)-K)n+\|X_n\|_\infty},
\end{align*}
since $F_n^s(\theta)\le n\log n-Kn$  if $\theta\in E_{n,s}^c$.
The lemma is proved.

\end{proof}

Let $\theta\in E_{n,s}$ and set $z_\mu=\phi_s(e^{\i\theta_\mu})$, $1\le\mu\le n$. Since,
\[
F_n^s(\theta)=\sum_{1\le \mu\neq\nu\le n}\log|\phi_s(e^{\i\theta_\mu})-\phi_s(e^{\i\theta_\nu})|+\sum_\mu\log|\phi_s'(e^{\i\theta_\mu})|
\]
we see that if $\theta\in E_{n,s}$ and $z_\mu=\phi_s(e^{\i\theta_\mu})$ are the corresponding points on $\gamma_s$, then
\begin{equation}\label{logenest}
\sum_{1\le \mu\neq\nu\le n}\log|z_\mu-z_\nu|^{-1}\le -n\log n+Kn+C(\gamma)n.
\end{equation}

The previous lemma will be helpful because if an $n$-tuple $\theta$ belongs to $E_{n,s}$, it will be rather close to $\alpha = (\frac{2\pi\mu}{n}+\sigma_n)_{\mu=1}^n$. We first prove a weak result in this direction.  

\begin{lemma}\label{lemma:suptmu} Fix a $K$ in \eqref{eq:En} independent of $s$. Set
\begin{equation}\label{epsilonn}
\epsilon_n=\sup_{0\le s\le 1}\sup_{\theta\in E_{n,s}}\max_{1\le \mu\le n}|t_\mu(\theta)|.
\end{equation}
Then $\epsilon_n\to 0$ as $n\to\infty$.

\end{lemma}

\begin{proof}
Let $\tau_\mu(\theta) = t_\mu(\theta) + \sigma_n(\theta)$, and let
\[ \delta_n = \sup_{0\leq s \leq 1} \sup_{\theta \in E_{n,s}} \max_{1\leq \mu \leq n} |\tau_\mu(\theta)|. \]
We will show that $\delta_n \to 0$ as $n\to \infty$. By definition, we have
\[ \frac{1}{n} \sum_\mu \tau_\mu(\theta) = \frac{1}{n} \sum_\mu t_\mu(\theta) + \sigma_n(\theta) = \sigma_n(\theta), \]
so $\delta_n \to 0$ implies that
\[ \sup_{0\leq s \leq 1} \sup_{\theta\in E_{n,s}} \sigma_n(\theta) \to 0 \]
as $n\to \infty$. Thus, $\delta_n \to 0$ leads to $\epsilon_n \to 0$ as $n\to \infty$.

First consider the probability measure
\[
\omega_k=\frac 1{n_k}\sum_{j=1}^{n_k}\delta_{e^{\i\theta_j^{(k)}}}.
\]
We will show that $\omega_k$ converges to the uniform measure $\d\theta/(2\pi)$ on $\T$. Let
\[
z_j^{(k)}=\phi_{s_k}(e^{\i\theta_j^{(k)}}), \quad 1\le j\le n_k,
\]
and 
\[
\mu_k=\frac 1{n_k}\sum_{j=1}^{n_k}\delta_{z_j^{(k)}},
\]
which is a probability measure on $\gamma_{s_k}$. The set $\cup_{0\le s\le 1}\gamma_s$ is contained in a compact set, and hence by picking a further subsequence, we can assume that
$s_k\to s\in [0,1]$ and $\mu_k\to\mu$, as $k\to\infty$, where $\mu$ is a probability measure on $\gamma_s$. The logarithmic energy of a probability measure $\mu$ on a curve $\gamma_s$
is defined by
\[
I[\mu]=\int_{\gamma_s}\int_{\gamma_s}\log |z-w|^{-1}\,d\mu(z) d\mu(w).
\]
If $\nu_s$ is the equilibrium measure on $\gamma_s$, then $I[\mu]\ge I[\nu_s]=\log(\mathrm{cap}(\gamma_s))=0$, and we have equality if and only if $\mu=\nu_s$, by uniqueness of the equilibrium measure, see \cite{SafTot}. Now,
\begin{align}\label{mukest}
\int_{\gamma_{s_k}}\int_{\gamma_{s_k}}\min(\log |z-w|^{-1},M)\,d\mu_k(z) d\mu_k(w)&\le\frac 1{n_k^2}\sum_{1\le j_1\neq j_2\le n_k}\log|z_{j_1}^{(k)}-z_{j_2}^{(k)}|^{-1}
+\frac 1{n_k^2}\sum_{j=1}^{n_k} M\notag\\
&\le\frac{-\log n_k +C(\gamma)+M}{n_k}
\end{align}
by \eqref{logenest} since $\theta^{(k)}\in E_{n_k,s_k}$. If we let $k\to\infty$ in \eqref{mukest}, we get
\[
\int_{\gamma_{s}}\int_{\gamma_{s}}\min(\log |z-w|^{-1},M)\,d\mu(z) d\mu(w)\le 0.
\]
Since the integrand is bounded from below, we can let $M\to\infty$ and get $I[\mu]\le 0=I[\nu_s]$ by the monotone convergence theorem. Thus, $\mu=\nu_s$. Since $\phi_s$ maps the equilibrium measure on $\T$ to the equilibrium measure on $\gamma_s$, we see that if we take the same subsequence in $\omega_k$, we have that $\omega_k$ converges to the uniform measure on $\T$.

Next, assume that $\limsup_{n\to\infty}\delta_n=\delta>0$. We will see that this leads to a contradiction. There is a subsequence $n_k\to\infty$ such that $\lim_{k\to \infty} \delta_{n_k} = \delta$.
Thus we can find a sequence $s_k\in [0,1]$ and $\theta^{(k)}\in E_{n_k,s_k}$ so that
\[
\max_{1\le \mu\le n}|\tau_\mu(\theta^{(k)})|\ge \delta/2
\]
for all sufficiently large $k$. There is a $\mu_k$ such that 
\[|\tau_{\mu_k}(\theta^{(k)})| = \max_{1\leq \mu \leq n_k} |\tau_\mu(\theta^{(k)})| .\]
After perhaps picking a further subsequence we can assume that $\tau_{\mu_k}(\theta^{(k)})>0$ for all $k$, or $\tau_{\mu_k}(\theta^{(k)})<0$ for all $k$. Assume the former, the other case being analogous. Then, $\tau_{\mu_k}(\theta^{(k)}) \geq \delta/2$, and thus
\begin{align}\label{eq:thetaLowerBound}
\theta_{\mu_k}^{(k)} = \frac{2\pi\mu_k}{n_k}+ \tau_{\mu_k}(\theta^{(k)}) \geq \frac{2\pi \mu_k}{n_k} + \frac{\delta}{2}.
\end{align}
Since $\theta^{(k)} \in H_{n_k}$, i.e.\ the $\theta^{(k)}_l$'s are ordered,
\begin{align}\label{eq:card}
\# \{ l: \theta_l^{(k)} \leq \frac{2\pi \mu_k}{n_k} + \frac{\delta}{2} \} \leq \mu_k. 
\end{align}
By possibly picking a further subsequence we can assume that 
\[ \frac{\mu_k}{n_k} \to a \in[0,1], \]
as $k\to \infty$. Since $\theta_{\mu_k}^{(k)} \leq 2\pi$, \eqref{eq:thetaLowerBound} implies $a\leq 1-\frac{\delta}{4\pi}$.
If $k$ is sufficiently large then $2\pi a +\frac{\delta}{4} \leq \frac{2\pi \mu_k}{n_k}+\frac{\delta}{2}$, and \eqref{eq:card} gives
\begin{align}\label{eq:card2}
    \frac{1}{n_k} \# \{ l: \theta_l^{(k)} \leq 2\pi a + \frac{\delta}{4} \} \leq \frac{\mu_k}{n_k}. 
\end{align}
By weak convergence of the probability measure $\omega_k$, the left side converges to $a+\frac{\delta}{8\pi} < 1$ as $k\to \infty$. Hence, letting $k\to \infty$ in \eqref{eq:card2} gives
\[ a+\frac{\delta}{8\pi} \leq a \]
which gives the desired contradiction. Consequently, $\delta_n$ and thus $\epsilon_n$ converge to zero as $n\to \infty$.

\end{proof}

\noindent\textbf{Remark.} The proof also shows that the empirical measure of $\{z_\mu\}_{\mu=1}^n$ converges weakly in probability to the equilibrium measure on $\gamma$, since by Lemma \ref{lemma:En}, $\lim_{n\to\infty}\mathbb{P}[E_{n,s}] = 1$ if we take $X_n=0$ and $K=C(\gamma)+1$.

The lemma can be used to obtain the following more precise bound on the size of the deviations $t_\mu$.

\begin{lemma}\label{lemma:tmu}
  There is a constant $C(\gamma,K)$, where $K$ is the constant in Lemma \ref{lemma:En}, such that $\sum_{\mu=1}^{n} t_\mu^2\le C(\gamma, K)$ for all
  $\theta\in E_{n,s}$.
\end{lemma}

\begin{proof}
We proceed as in the proof of Lemma 7.6 in \cite{Joh88}; but the structure of the proof ultimately goes back to \cite{Pom69}.
Set $\beta_\mu = \alpha_\mu + \tau t_\mu$, with $\alpha_\mu = 2\pi\mu/n+\sigma_n$ and $\sum_\mu t_\mu =0$ as above. Let 
$$\psi_n(\tau) = F_n^s(\beta) = \sum_{1\leq \mu,\nu \leq n} \log \big\lvert \frac {\phi_s(e^{\i \beta_{\mu}})-\phi_s(e^{\i \beta_{\nu}})}{e^{\i \beta_{\mu}}-e^{\i \beta_{\nu}}} \big\rvert + \sum_{1\leq \mu\neq\nu \leq n} \log\lvert e^{\i \beta_{\mu}}-e^{\i \beta_{\nu}}\rvert.$$ 
Then, by \eqref{Grunskys},
\[ \psi_n'(\tau) = \sum_{\mu\neq \nu} \cot\Big(\frac{\beta_\mu-\beta_\nu}{2}\Big)\frac{t_\mu-t_\nu}{2} - 2 \Im \sum_{k,l\geq 1} l s^{k+l}a_{kl}(s) \Big(\sum_\mu e^{-\i k\beta_\mu}\Big) \Big(\sum_\nu t_\nu e^{-\i l\beta_\nu}\Big).
\]
The first sum evaluated at $\tau=0$ equals
\[
\sum_{k=1}^{n-1} \cot(\pi k/n)\sum_{\mu=1}^n t_\mu = 0,
\]
and the second sum evaluated at $\tau=0$ equals
\begin{align*}
    2 \Im \sum_{k,l\geq 1} l s^{nk+l}a_{nk,l}(s)e^{-\i nk\sigma_n} \Big(\sum_\nu t_\nu e^{-\i l\alpha_\nu}\Big) 
\end{align*}
because $\sum_{\mu=1}^n e^{-\i k\alpha_\mu} = 0$ unless $k$ is a multiple of $n$. It follows from  \eqref{eq:coeffdecay}, that
\[ 
|\psi_n'(0)|\le C(\gamma) n \sum_{k,l\geq 1} l(nk)^{-6-\epsilon}l^{-2-\epsilon}\le C(\gamma).
\]
Differentiating again gives
\begin{align*}
\psi_n''(\tau) = -\sum_{\mu\neq\nu} \frac{(t_\mu-t_\nu)^2}{4\sin^2((\beta_\mu-\beta_\nu)/2)} + A_\tau + B_\tau, 
\end{align*}
where
\[
A_\tau = 2 \Re \sum_{\mu,\nu} \sum_{k,l\geq 1} k^2s^{k+l}a_{kl}(s) t_\mu^2 e^{-\i k\beta_\mu} e^{-\i l\beta_\nu} ,\quad B_\tau = 2 \Re \sum_{\mu,\nu} \sum_{k,l\geq 1} kls^{k+l}a_{kl}(s)  t_\mu t_{\nu}e^{-\i k\beta_\mu} e^{-\i l\beta_\nu}.
\]
Thus,
\begin{align}\label{tmuest1}
&\int_0^1 \sum_{\mu\neq \nu} \frac{(t_\mu-t_\nu)^2}{4\sin^2((\beta_\mu-\beta_\nu)/2)}(1-\tau) \d\tau=-\int_0^1(1-\tau)\psi_n''(\tau)\,d\tau+\int_0^1(A_\tau+B_\tau)(1-\tau)\,d\tau\notag\\
&=F_n^s(\alpha)-F_n^s(\theta) + \psi_n'(0)+\int_0^1 (A_\tau+B_\tau)(1-\tau)\,d\tau \le Kn+C(\gamma)+\int_0^1 (|A_\tau|+|B_\tau|)(1-\tau)\,d\tau,
\end{align}
since $\theta\in E_{n,s}$.
To bound $|A_\tau|$, we note that by \eqref{epsilonn}, 
\begin{align}\label{betasum}
\Big| \sum_\nu e^{-\i l \beta_\nu} \Big| &\le \Big| \sum_\nu (e^{-\i l \beta_\nu} -e^{-\i l\alpha_\nu})\Big| (1-\mathbbm{1}(n|l))+n\mathbbm{1}(n|l)\notag\\
&\leq l \sum_\nu |t_\nu| +n\mathbbm{1}(n|l)\leq n(l\epsilon_n +\mathbbm{1}(n|l)),
\end{align}
where $n|l$ indicates that $n$ divides $l$. 
Thus, by \eqref{eq:coeffdecay} and the fact that $|a_{kl}(s)|\le |a_{kl}|$,
\begin{align*}
|A_\tau| &\leq 2\big(\sum_\mu t_\mu^2 \big)\big(\sum_{k,l\geq 1}k^2|a_{kl}(s)| \big|\sum_\nu e^{-\i l \beta_\nu} \big|\big)\\
&\le 2\big(\sum_\mu t_\mu^2\big)\big[\big(\sum_{k,l\geq 1}k^2l|a_{kl}|n\epsilon_n+n\sum_{k,l\geq 1}k^2|a_{k,nl}|\big]\leq C(\gamma)n\epsilon_n' \big(\sum_{\mu} t_{\mu}^2\big),
\end{align*}
where 
\begin{equation}\label{epsnprime}
\epsilon_n'=\epsilon_n+1/n^2.
\end{equation}

To bound $|B_\tau|$ we divide it into two parts,
\[B_\tau^1= 2 \Re \sum_{1\leq k,l\leq q} \sum_{\mu,\nu} kla_{kl}  t_\mu t_{\nu}e^{-\i k\beta_\mu} e^{-\i l\beta_\nu}, \quad B_\tau^2= 2 \Re \sum_{\substack{k\geq 1, l\geq1  \\ \mathrm{max}(k,l) >q} }\sum_{\mu,\nu} kla_{kl}  t_\mu t_{\nu}e^{-\i k\beta_\mu} e^{-\i l\beta_\nu}\]
where $q\in\N$ will be fixed later.
By the strengthened Grunsky inequality, Lemma \ref{lemma:grunskyineq}, there exists a $\kappa<1$ such that
\[|B_\tau^1|  \leq 2 \kappa \sum_{k= 1}^q k\Big|\sum_\mu e^{-\i k\beta_\mu}t_\mu\Big|^2.\]
Set $\rho = \kappa^{1/(2q)}$. Then $\kappa\leq \rho^{2k}$, for $1\leq k \leq q$, so
\[|B_\tau^1|  \leq 2 \sum_{k\geq 1} k \rho^{2k} \Big|\sum_\mu e^{-\i k\beta_\mu}t_\mu\Big|^2
    = \sum_{\mu,\nu} \big[(t_\mu^2 + t_\nu^2 - (t_\mu-t_\nu)^2) \sum_{k\geq 1} k\rho^{2k}e^{-\i k(\beta_\mu-\beta_\nu)}\big],\]
by writing $2t_\mu t_\nu = t_\mu^2 + t_\nu^2 - (t_\mu-t_\nu)^2$. 

If we use \eqref{betasum}, we get
\begin{align*}
    |B_\tau^1| &\leq 2\big(\sum_\mu t_\mu^2\big)\big(\sum_{k\ge 1}k\rho^{2k}\big|\sum_\nu e^{\i k\beta_\nu}\big|\big)+
    \sum_{\mu,\nu}(t_\mu-t_\nu)^2\left|\frac{\rho^2e^{\i(\beta_\mu-\beta_\nu)}}{(1-\rho^2e^{\i(\beta_\mu-\beta_\nu)})^2}\right|\\
    &\le 2\big(\sum_\mu t_\mu^2\big)\big(n\epsilon_n\sum_{k\ge 1}k^2\rho^{2k}+n^2\sum_{k\ge 1}k\rho^{2nk}\big)+\sum_{\mu,\nu}\frac{(t_\mu-t_\nu)^2}
    {a+4\sin^2\frac 12(\beta_\mu-\beta_\nu)}\\
    &\le C\big(\sum_\mu t_\mu^2\big)\left(\frac{n\epsilon_n}{(1-\rho)^3}+\frac{n^2\rho^{2n}}{(1-\rho^n)^2}\right)+\frac 1{a+4}\sum_{\mu,\nu}\frac{(t_\mu-t_\nu)^2}
    {\sin^2\frac 12(\beta_\mu-\beta_\nu)}.
\end{align*}
where $a=(1/\rho-\rho)^2$. 
Next, \eqref{eq:coeffdecay} followed by the Cauchy-Schwarz inequality gives
\begin{align*}
    |B_\tau^2| &\le C(\gamma)\Big(\sum_{k\ge 1}\frac 1{k^{2}}\Big|\sum_\mu t_\mu e^{-\i k\beta_\mu}\Big|\Big)\Big(\sum_{l\ge q}\frac 1{l^{3+\epsilon}}\Big|\sum_\nu t_\nu e^{-\i l\beta_\nu}\Big|\Big)\\ &\le
     \frac{C(\gamma)}{q^{1+\epsilon}} \Big(\sum_{k\ge 1}\frac 1{k^{2}}\Big|\sum_\mu t_\mu e^{-\i k\beta_\mu}\Big|\Big)^2\le\frac{C(\gamma)}{q^{1+\epsilon}} 
 \sum_{k\ge 1}\frac 1{k^{2}}\Big|\sum_\mu t_\mu e^{-\i k\beta_\mu}\Big|^2.
    \end{align*}
Again, writing $t_\mu t_\nu = \frac{1}{2}(t_\mu^2 + t_\nu^2 - (t_\mu-t_\nu)^2)$, and using \eqref{betasum} gives
\begin{align*}
    |B_\tau^2| &\le  \frac{C(\gamma)}{q^{1+\epsilon}}\sum_{k\ge 1}\frac 1{k^{2}}\sum_{\mu,\nu}\frac 12(t_\mu^2+t_\nu^2-(t_\mu-t_\nu)^2)e^{-\i k(\beta_\mu-\beta_\nu)}\\
    &\le \frac{C(\gamma)}{q^{1+\epsilon}}\Big[\sum_{k\ge 1}\frac 1{k^{2}}\Big(\sum_\mu t_\mu^2\Big)\Big|\sum_\mu e^{\i k\beta_\nu}\Big|+\Big(\sum_{k\ge 1}\frac 1{k^{2}}
    \Big)\sum_{\mu,\nu}(t_\mu-t_\nu)^2\Big]\\
    &\le  \frac{C(\gamma)}{q^{1+\epsilon}}\Big[n\epsilon_n'\Big(\sum_\mu t_\mu^2\Big)+\sum_{\mu,\nu}(t_\mu-t_\nu)^2\Big].
    \end{align*}
Inserting these estimates into \eqref{tmuest1} we find
\begin{align*}
&\big(\frac 14-\frac 1{a+4}\big)\int_0^1\sum_{\mu\neq\nu}\frac{(t_\mu-t_\nu)^2}{\sin^2\frac 12(\beta_\mu-\beta_\nu)}(1-\tau)\,d\tau-\frac {C(\gamma)}{q^{1+\epsilon}}\sum_{\mu,\nu}(t_\mu-t_\nu)^2\\
&\le Kn+C(\gamma)+C(\gamma)n\epsilon_n''\big(\sum_\mu t_\mu^2\big),
\end{align*}
where
\[
\epsilon_n''=\epsilon_n'+\frac{\epsilon_n}{(1-\rho)^3}+\frac{n\rho^{2n}}{(1-\rho^n)^2}.
\]
Observe that 
\[\sum_{\mu\neq\nu}(t_\mu-t_\nu)^2 = 2n\sum_\mu t_\mu^2-2\sum_{\mu,\nu}t_\mu t_\nu = 2n\sum_\mu t_\mu^2 \]
since $\sum_{\mu} t_\mu =0$. Using that $\frac{a}{a+4} = (\frac{1-\rho^2}{1+\rho^2})^2$, we obtain the inequality
\[
\left(\frac 12\big(\frac{1-\rho^2}{1+\rho^2}\big)^2-\frac {C(\gamma)}{q^{1+\epsilon}}\right)\big(\sum_\mu t_\mu^2\big)\le K+\frac{C(\gamma)}n+C(\gamma)\epsilon_n''\big(\sum_\mu t_\mu^2\big).
\]
Since, $\rho=\kappa^{1/2q}<1$, we see that 
\[
\frac 12\big(\frac{1-\rho^2}{1+\rho^2}\big)^2-\frac {C(\gamma)}{q^{1+\epsilon}}\ge\frac 14(1-e^{-\frac 1q\log\frac 1\kappa})-\frac {C(\gamma)}{q^{1+\epsilon}}\ge C_0(\gamma,K)>0,
\]
if we pick $q$ large enough. Since $\epsilon_n''\to0$ as $n\to\infty$ we see that if we pick $n$ large enough depending on $\gamma$ and $K$ then $C(\gamma)\epsilon_n''\le \frac 12C_0(\gamma,K)$.
This proves the lemma.

% Next assume that $t_\mu > s$. If $\theta\in E_n$ then $\theta_{\mu+1}>\theta_\mu$ so $t_{\mu+1}>t_\mu-2\pi/n$ and $t_\nu \geq s/2$ for $\mu\leq \nu\leq \mu+[\frac{sn}{4\pi}]$. It follows that
% \[ C\geq \sum_\mu t_\mu^2 \geq \frac{s^2}{4}\Big[\frac{sn}{4\pi}\Big] \geq cs^3n \]
% which proves that $s<C/n^{1/3}$. The proof of the lower bound is similar.
\end{proof}

\noindent\textbf{Remark.} This result can be used to obtain a rate of convergence of the empirical measure $\lambda_n = \frac{1}{n}\sum_{\mu} \delta_{\theta_\mu}$ to the uniform distribution on the unit circle in the $L_2$ Wasserstein metric. We take $X_n=0$, $s=1$, and $K=C(\gamma)+1$ in Lemma \ref{lemma:En}.
Then, with $\nu_n= \frac{1}{n} \sum_{\mu} \delta_{\alpha_\mu}$,
\[ W_2(\lambda_n,\nu_n)^2 \leq \frac{1}{n} \sum_{\mu=0}^{n-1} (\theta_\mu-\alpha_\mu)^2 = \frac{1}{n} \sum_{\mu=0}^{n-1} t_\mu^2 \leq \frac{C(\gamma)}{n} \]
if $\theta\in E_{n,1}$, by Lemma \ref{lemma:tmu}. Thus,
\[ \mathbb{P}[W_2(\lambda_n,\nu_n) > \tfrac{C(\gamma)}{\sqrt{n}}] \leq \mathbb{P}[E_{n,1}^c] < e^{-n} \]
by Lemma \ref{lemma:En}, so the Borel-Cantelli lemma gives that $W_2(\lambda_n,\nu_n) \le \tfrac{C(\gamma)}{\sqrt{n}}$ for $n$ sufficiently large, almost surely. Moreover, $W_2(\nu_n,\tfrac{\d\theta}{2\pi}) \leq \sqrt{2} W_1(\nu_n,\tfrac{\d\theta}{2\pi})^{1/2} \leq (\frac{2\pi}{n})^{1/2}$, by using the dual representation of $W_1$. We obtain 
\[ W_2(\lambda_n,\frac{\d\theta}{2\pi}) \le \tfrac{C(\gamma)}{\sqrt{n}}\ \mathrm{a.s.}\]

\subsection{The change of variables}
Let $g_s$ be a function on $\gamma_s$. Recall that we define $G_s(\theta)=g_s\circ\phi_s(e^{\i\theta})$, and $\g(s)$ as in \eqref{gvector} but with $g_s\circ\phi_s$ instead of
$g\circ\phi$. Also, as in Lemma \ref{lem:inteqnsol}, we let
\[
\h(s)=\frac 2\beta L(I+K(s))^{-1}\g(s),
\]
and 
\[
H_s(\theta)=2\begin{pmatrix}
    \mathbf{x_\theta} \\ \mathbf{y_\theta}
\end{pmatrix}^t\h(s).
\]
Since $|a_{kl}(s)|\le |a_{kl}|$, the constant $A$ in \eqref{eq:coeffdecay} is independent of $s$, and we also have that $\|K(s)\|\le\kappa<1$. Hence, by \eqref{HLip},
\begin{equation}\label{Hsbound}
\|H_s\|_{4,\alpha}\le C(\gamma)\|G_s\|_{4,\alpha}.
\end{equation}
Let $h_k(s)$ be the $k$:th complex Fourier coefficient of $H_s$.

Consider now the expectation
\[ 
\E_n^{\beta,s}\Big[\exp \Big( \sum_\mu G_s(\theta_\mu) \Big)\Big] = \frac{1}{D_n^{\beta,s}[1]n!} \int_{[0,2\pi]^n} \exp \Big( \frac{\beta}{2} F_n^s(x)+  \sum_\mu G_s(x_\mu) 
+(1-\tfrac{\beta}{2})\log|\phi_s'(e^{\i x_\mu})| \Big)\, \d^n x.
\]
In this integral we make the change of variables $x_\mu = \theta_\mu-\frac{1}{n}H_s(\theta_\mu)$. The domain of integration, $[0,2\pi]^n$, is unchanged by periodicity. 
It follows from \eqref{Grunskys} and \eqref{Fns} that
\[
F_n^s(x)=\sum_{\mu\neq\nu}\log\big|2\sin\frac{x_\mu-x_\nu}{2}\big|-\Re\sum_{k,l\ge 1}a_{kl}(s)\Big(\sum_\mu e^{-\i kx_\mu}\Big)\Big(\sum_\nu e^{-\i lx_\nu}\Big).
\]
Write 
\[
L_s(x)=\log|\phi_s'(e^{\i x})|,
\]
and let $x_\mu=x_\mu(\tau):=\theta_\mu-\frac\tau n H_s(\theta_\mu)$, so that $x_\mu'(\tau)=-\frac 1nH_s(\theta_\mu)$ and $x_\mu''(\tau)=0$. Define
\begin{align*}
f_n(\tau)&= \frac{\beta}{2} \sum_{\mu\neq\nu} \Big( 
 \log\big|2\sin\frac{x_\mu(\tau)-x_\nu(\tau)}2\big|-\Re\sum_{k,l\ge 1}a_{kl}(s)\Big(\sum_\mu e^{-\i kx_\mu(\tau)}\Big)\Big(\sum_\nu e^{-\i lx_\nu(\tau)}\Big) \Big) \\
&+\sum_\mu G_s(x_\mu(\tau))+(1-\tfrac \beta 2)\sum_\mu L_s(x_\mu(\tau)).
\end{align*}
We have the Taylor expansions
\begin{align*}
&f_n(1)=f_n(0)+f_n'(0)+\frac 12 f_n''(0)+\frac 12\int_0^1(1-\tau)^2f_n^{(3)}(\tau)\,d\tau,\\
&\log \big(1-\frac{H_s'(\theta_\mu)}{n} \big) = -\frac{1}{n}H_s'(x) -\frac 1{n^2}\sum_\mu\int_0^1(1-\tau)\big(1-\frac{\tau}nH_s'(\theta_\mu)\big)^{-2}H_s'(\theta_\mu)^2\,d\tau.
\end{align*}
Let
\begin{align}\label{eq:UTR}
R_n^s(\theta) =& -\beta \Re \sum_{k,l\geq 1} ika_{kl}(s) h_k(s) \Big(\sum_{\mu=1}^n e^{-\i l\theta_\mu} \Big)-\frac{\beta}{4n}\sum_{\mu\neq\nu} \cot\Big(\frac{\theta_\mu-\theta_\nu}{2} \Big) \big(H_s(\theta_\mu)-H_s(\theta_\nu)\big)+\sum_{\mu=1}^n G_s(\theta_\mu), \nonumber \\
T_n^s(\theta) =& \beta \Re \sum_{k,l\geq 1} ika_{kl}(s) \Big(h_k(s)-\frac{1}{n}\sum_{\mu=1}^n e^{-\i k\theta_\mu}H_s(\theta_\mu)\Big) \sum_{\nu=1}^n e^{-\i l\theta_\nu}, \nonumber \\
U_n^s(\theta) =& -\frac{1}{n}\sum_{\mu=1}^n \big(G_s'(\theta_\mu)+ (1-\tfrac{\beta}{2})L_s'(\theta_\mu)\big) H(\theta_\mu) -\frac{\beta}{16n^2}\sum_{\mu,\nu} \frac{(H_s(\theta_\mu)-H_s(\theta_\nu))^2}{\sin^2((\theta_\mu-\theta_\nu)/2)} \nonumber \\
&-\frac{1}{n}\sum_{\mu=1}^n H_s'(\theta_\mu)
+\frac{\beta}{2}\Re \sum_{k,l\geq 1} kla_{kl}(s) \Big(\frac{1}{n}\sum_{\mu=1}^n e^{-\i k\theta_\mu}H_s(\theta_\mu)\Big)\Big(\frac{1}{n}\sum_{\nu=1}^n e^{-\i l\theta_\nu}H_s(\theta_\nu)\Big) \nonumber \\
&+\frac{\beta}{2} Re \sum_{k,l\geq 1} k^2a_{kl}(s) \Big(\frac{1}{n}\sum_{\mu=1}^n e^{-\i k\theta_\mu}H_s^2(\theta_\mu)\Big)\Big(\frac{1}{n}\sum_{\nu=1}^n e^{-\i l\theta_\nu}\Big). 
\end{align}
and
\begin{align}\label{Sn}
S_n^s(\theta)&=\frac 12\int_0^1(1-\tau)^2f_n^{(3)}(\tau)\,d\tau+\frac 1{2n^2}\sum_\mu G_s''(\theta_\mu)H_s(\theta_\mu)^2+(1-\tfrac \beta2)\frac 1{2n^2}\sum_\mu L_s''(\theta_\mu)H_s(\theta_\mu)^2\notag\\
&-\frac 1{n^2}\sum_\mu\int_0^1(1-\tau)\big(1-\frac{\tau}nH_s'(\theta_\mu)\big)^{-2}H_s'(\theta_\mu)^2\,d\tau.
\end{align}
The change of variables gives, after some computation,
\begin{equation}\label{Expchvar}
\E_n^{\beta,s}\Big[\exp\big(\sum_\mu G_s(\theta_\mu)\big)\Big]=\E_n^{\beta,s}\big[\exp\big((R_n^s+T_n^s+U_n^s+S_n^s)(\theta)\big)\big].
\end{equation}
Also,
\begin{align}\label{f3n}
f_n^{(3)}(\tau)&=-\frac 1{4n^3}\sum_{\mu\neq\nu}\frac{\cos\frac{x_\mu-x_\nu}2}{\sin^3\frac{x_\mu-x_\nu}2}\big(H_s(\theta_\mu)-H_s(\theta_\nu)\big)^3\notag\\
&-\frac 2{n^3}\Re\sum_{k,l\ge 1}(-\i)k^3a_{kl}(s)\Big(\sum_\mu e^{-\i kx_\mu}H_s(\theta_\mu)^3\Big)\Big(\sum_\nu e^{-\i lx_\nu}\Big)\notag\\
&-\frac 6{n^3}\Re\sum_{k,l\ge 1}(-\i)k^2la_{kl}(s)\Big(\sum_\mu e^{-\i kx_\mu}H_s(\theta_\mu)^2\Big)\Big(\sum_\nu e^{-\i lx_\nu}H_s(\theta_\nu)\Big)\notag\\
&-\frac 1{n^3}\sum_\mu G_s^{(3)}(x_\mu)H_s(\theta_\mu)^3-(1-\tfrac\beta 2)\frac 1{n^3}\sum_\mu L_s^{(3)}(x_\mu)H_s(\theta_\mu)^3,
\end{align}
where $x_\mu=\theta_\mu-\tfrac\tau nH_s(\theta_\mu)$.

It follows from \eqref{Grunsky2} and \eqref{eq:coeffdecay} that there is a constant $C(\gamma)$ such that
\begin{equation}\label{Lsbound}
\big\|\frac{d^r}{d\theta^r}L_s\big\|_\infty\le C(\gamma),
\end{equation}
for $0\le r\le 3$. Using \eqref{eq:coeffdecay}, \eqref{Hsbound} and \eqref{Lsbound}, we obtain the estimate
\[
|f_n^{(3)}(\tau)|\le\frac 1n\big(C\|H_s\|_\infty^3+CA\|H_s\|_\infty^3+\|G_s^{(3)}\|_\infty\|H_s\|_\infty^3)\le \frac {C(\gamma, \|G_s\|_{4,\alpha})}n.
\]
From this and \eqref{Sn} we see that
\begin{equation}\label{Snest}
\|S_n^s\|_\infty\le \frac {C(\gamma, \|G_s\|_{4,\alpha})}n.
\end{equation}
Also from \eqref{eq:UTR}, we obtain the estimates
\begin{equation}\label{Unest}
\|U_n^s\|_\infty\le C(\gamma, \|G_s\|_{4,\alpha}),
\end{equation}
and
\begin{equation}\label{RnTnest}
\|R_n^s+T_n^s\|_\infty\le C(\gamma, \|G_s\|_{4,\alpha})n.
\end{equation}
Combining \eqref{Snest}, \eqref{Unest} and \eqref{RnTnest}, we see that there is a constant $C_1(\gamma, \|G_s\|_{4,\alpha})$ such that
\begin{equation}\label{Xnest}
\|R_n^s+T_n^s+U_n^s+S_n^s\|_\infty\le C_1(\gamma, \|G_s\|_{4,\alpha})n.
\end{equation}
In Lemma \ref{lemma:En} we now choose $K=C(\gamma)+C_1(\gamma, \|G_s\|_{4,\alpha})+1$, where $C(\gamma)$ is the constant in the lemma. Below $E_{n,s}$ will denote the
set obtained from Lemma \ref{lemma:En} with this choice of $K$.
This gives us the estimate
\begin{equation}\label{IntEnsest}
\Big|\E_n^{\beta,s}\Big[\exp \Big( \sum_\mu G_s(\theta_\mu) \Big)\Big] -\E_n^{\beta,s}\Big[\exp \big((R_n^s+T_n^s+U_n^s+S_n^s)(\theta)\big)\mathbbm{1}_{E_{n,s}}\Big]\Big|\le e^{-n}.
\end{equation}
It follows from Lemma \ref{lemma:tmu} that there is a constant $C(\gamma, \|G_s\|_{4,\alpha})$ such that if $\theta\in E_{n,s}$ then
\begin{equation}\label{Sumtmuest}
\sum_\mu t_\mu^2=\sum_\mu t_\mu(\theta)^2\le C(\gamma, \|G_s\|_{4,\alpha}).
\end{equation}
Hence, if $\theta\in E_{n,s}$ and $l$ does not divide $n$, then
\[
\Big|\sum_\nu e^{\i l\theta_\nu}\Big|=\Big|\sum_\nu (e^{\i l\theta_\nu}-e^{\i l\alpha_\nu})\Big|\le l\sum_\nu|t_\nu|\le l\sqrt{n}\Big(\sum_\nu t_\nu^2\Big)^{1/2}\le
C(\gamma, \|G_s\|_{4,\alpha}) l\sqrt{n}.
\]
Note that if $l\ge 1$ divides $n$, then $l\ge n$ and the sum is always $\le n$, so we always have an estimate
\begin{equation}\label{Sumthetanu}
\Big|\sum_\nu e^{\i l\theta_\nu}\Big|\le C(\gamma, \|G_s\|_{4,\alpha}) l\sqrt{n}.
\end{equation}
A simple Riemann sum estimate gives
\begin{equation}\label{RSest}
\Big|\frac 1n\sum_\mu f(e^{2\pi\i \mu/n})-\frac 1{2\pi}\int_0^{2\pi}f(e^{\i t})\,dt\Big|\le\frac{\|f'\|_\infty}n.
\end{equation}
It follows from this estimate and \eqref{Hsbound} that
\begin{equation}\label{hkest}
\Big|h_k(s)-\frac 1n\sum_\mu e^{-\i k\alpha_\mu}H_s(\alpha_\mu)\Big|\le C\frac{k\|H_s\|_\infty+\|H_s'\|_\infty}n\le C(\gamma, \|G_s\|_{4,\alpha})\frac kn.
\end{equation}

\begin{lemma}\label{lemma:Tn}
There is a constant $C(\gamma, \|G_s\|_{4,\alpha})$ such that $|T_n^s(\theta)| \leq C(\gamma, \|G_s\|_{4,\alpha})$ for all $n\geq 1$ and $\theta \in E_{n,s}$.
\end{lemma}

\begin{proof}
We have the estimate
\begin{align}\label{eq:Tn2}
    \Big|\frac{1}{n}\sum_\mu \big(e^{-\i k\theta_\mu}H(\theta_\mu)-e^{-\i k\alpha_\mu}H(\alpha_\mu) \big) \Big| &\leq C\|H_s\|_{4,\alpha} \frac{k}{n}\sum_\mu |t_\mu| 
    \leq C\|H_s\|_{4,\alpha}   \frac{k}{\sqrt{n}} \Big(\sum_\mu t_\mu^2 \Big)^{1/2} \notag\\
    &\leq  C(\gamma, \|G_s\|_{4,\alpha})\frac{k}{\sqrt{n}},
\end{align}
where we used \eqref{Hsbound}. This estimate, together with \eqref{eq:coeffdecay}, \eqref{Sumthetanu},  and \eqref{hkest}, gives
\[
|T_n^s(\theta)|\le C(\gamma, \|G_s\|_{4,\alpha})\sum_{k,l\ge 1}k^2l|a_{kl}(s)|\le C(\gamma, \|G_s\|_{4,\alpha}),
\]
and we are done.
\end{proof}

\begin{lemma}\label{lemma:Rn}
There is a constant $C(\gamma, \|G_s\|_{4,\alpha})$ such that $|R_n^s(\theta)| \leq C(\gamma, \|G_s\|_{4,\alpha})$ for all $n\geq 1$ and $\theta \in E_{n,s}$.
\end{lemma}

\begin{remark}\label{rem:greg} 
This proposition is the origin of the assumption $g\in C^{4+\epsilon}$. It is possible to only assume $g\in C^{1+\epsilon}$, by following the techniques of proof of \cite{Joh88}, and prove the analogues of Lemma 1.2, 2.1, and 2.2 in \cite{Joh88}. We chose to present a shorter proof here for the sake of brevity and simplicity.
\end{remark}

\begin{proof}
It follows from \eqref{Grunsky} by differentiation that
\[
\Re\frac{\i e^{\i \theta}\phi'(e^{\i\theta})}{\phi(e^{\i\theta})-\phi(e^{\i\omega})}=-\frac 12\cot(\tfrac{1}{2}(\omega-\theta))+\Re\sum_{k,l\ge 1}\i ka_{kl}e^{-\i k\theta-\i l\omega}.
\]
and hence we can write the integral equation \eqref{inteq} as
\begin{align}\label{eq:G}
G_s(\omega) = -\frac{\beta}{2} p.v. \int_0^{2\pi} \cot(\tfrac{1}{2}(\omega-\theta))H_s(\theta)\frac{\d\theta}{2\pi}+\beta \Re \sum_{k,l\geq 1} \i ka_{kl} h_k(s)e^{-\i l \omega}
\end{align}
It follows that
\begin{equation}\label{eq:Rn}
    R_n^s(\theta) = -\frac{\beta}{2} \sum_\mu \Big[ \int_0^{2\pi} \cot(\tfrac{1}{2}(\theta_\mu-\theta))H_s(\theta)\frac{\d\theta}{2\pi}+\frac{1}{2n}\sum_{\nu:\nu\neq\mu} \cot(\tfrac{1}{2}(\theta_\mu-\theta_\nu))(H_s(\theta_\mu)-H_s(\theta_\nu)) \Big]. 
\end{equation}
We recognize the integral above as the Hilbert transform giving the conjugate function on $\T$, so 
\begin{equation}\label{conj}
\int_0^{2\pi} \cot(\tfrac{1}{2}(\theta_\mu-\theta))H_s(\theta) \frac{\d\theta}{2\pi} = \sum_{k\in\Z}-\i \sgn(k) h_k(s)e^{\i k\theta_\mu}.
\end{equation}
Note that for $k\ge 1$
\[
 \cot(\tfrac{1}{2}(x-y))(e^{\i kx}-e^{\i ky})=\i(e^{\i x}+e^{\i y})\sum_{j=0}^{k-1}e^{\i(jx+(k-1-j)y}=\i\Big(e^{\i kx}+e^{\i ky}+2\sum_{j=1}^{k-1}e^{\i(jx+(k-j)y}\Big)
 \]
Changing the sign of $k,x$ and $y$, we see that if $|k|\ge 1$, then
\[
 \cot(\tfrac{1}{2}(x-y))(e^{\i kx}-e^{\i ky})=\i\sgn(k)\Big(e^{\i kx}+e^{\i ky}+2\sum_{j=1}^{|k|-1}e^{\i\sgn(k)(jx+(|k|-j)y)}\Big).
 \]
If we use this and \eqref{conj}, we obtain
\[
R_n^s(\theta)=\frac \beta 2\sum_{k\in\Z}\i\sgn(k)h_k(s)\Big(\frac{|k|}n\sum_\mu e^{\i k\theta_\mu}+\frac 1n\sum_{\mu,\nu}\sum_{j=1}^{|k|-1}e^{\i\sgn(k)(j\theta_\mu+(|k|-j)\theta_\nu)}\Big).
\]
We can now use \eqref{Sumthetanu} to see that if $\theta\in E_{n,s}$ then
\[
|R_n^s(\theta)|\le C(\gamma,\|G_s\|_{4,\alpha})\sum_{k\in\Z}|h_k(s)|\big(\frac{k^2}{\sqrt{n}}+|k|^3\big).
\]
Here we can use \eqref{Hsbound} and a standard bound of Fourier coefficients to obtain the estimate
\[
|h_k(s)|\le C\|H_s\|_{4,\alpha}\frac 1{|k|^{4+\alpha}}\le C(\gamma,\|G_s\|_{4,\alpha})\frac 1{|k|^{4+\alpha}}.
\]
From the last two estimates we can now conclude that $|R_n^s(\theta)| \leq C(\gamma, \|G_s\|_{4,\alpha})$.
\end{proof}

\begin{remark}\label{rem:inteq}
Equations \eqref{eq:G} and \eqref{eq:Rn} are the reason for picking $h$ to be the solution of the integral equation \eqref{inteq}. Thanks to this choice we see that the first two terms of $R_n^s$ in \eqref{eq:UTR} almost cancel the third one.
\end{remark}

Combining the estimates \eqref{Snest}, \eqref{Unest}, \eqref{IntEnsest}, Lemma \ref{lemma:Tn}, and Lemma \ref{lemma:Rn}, we see that we have proved the bound \eqref{expbound}
in Theorem \ref{smainthm}. It remains to prove the limit \eqref{eq:smain} for each fixed $s$. 

\subsection{Computing the limit}

We will use the following simple lemma.
\begin{lemma} \label{lem:limit}
\cite{Sim04}
Let $E_{n,s}$ be the set in Lemma \ref{lemma:En} with the choice of $K$ above. 
Assume that there is a constant $C_0=C(\gamma, G)$ such that $\sup_n \| X_n \mathbbm{1}_{E_{n,s}}\|_\infty \le C_0$, and that $\lim_{n\to\infty} \E_n^{\beta,s}[|X_n|\mathbbm{1}_{E_{n,s}}] =0$. Then 
\[ \lim_{n\to\infty} \E_n^{\beta,s}[e^{X_n}\mathbbm{1}_{E_{n,s}}] = 1.\]
\end{lemma}
\begin{proof}
By Lemma \ref{lemma:En}, $\lim_{n\to\infty} \mathbb{P}_n^{\beta,s}[E_{n,s}^c] =0$, so
\[ \lim_{n\to\infty}|\E_n^{\beta,s}[e^{X_n}\mathbbm{1}_{E_{n,s}}]-1| \leq \lim_{n\to\infty}| \E_n^{\beta,s}[(e^{X_n}-1)\mathbbm{1}_{E_{n,s}}]|+\mathbb{P}_n^{\beta,s}[E_{n,s}^c]| = \lim_{n\to\infty} | \E_n^{\beta,s}[(e^{X_n}-1)\mathbbm{1}_{E_{n,s}}]|.\]
But
\[| \E_n^{\beta,s}[(e^{X_n}-1)\mathbbm{1}_{E_{n,s}}]| \leq \E_n^{\beta,s}[|X_n|e^{|X_n|}\mathbbm{1}_{E_{n,s}}] \leq e^{C_0}\E_n^{\beta,s}[|X_n|\mathbbm{1}_{E_n}],\]
which goes to zero as $n\to\infty$ by assumption.
\end{proof}

Let
\begin{align}\label{us}
u_s(\theta,\omega)&=-G_s'(\theta)H_s(\theta)-(1-\tfrac \beta 2)L_s'(\theta)H(s)-H_s'(\theta)-\frac{\beta}{16}\frac{(H_s(\theta)-H_s(\omega))^2}{\sin^2\frac{\theta-\omega}2}\notag\\
&+\frac{\beta}2\Re\sum_{k,l\ge 1}kla_{kl}(s)e^{-\i k\theta-\i l\omega}H_s(\theta)H_s(\omega)+\beta \Re\sum_{k,l\ge 1}k^2a_{kl}(s)e^{-\i k\theta-\i l\omega}H_s(\theta)^2,
\end{align}
and let
\[
d\lambda_{n,\theta}(t)=\frac 1n\sum_\mu\delta(t-\theta_\mu),
\]
be the empirical measure. Note that
\[
U_n^s(\theta)=\int_{[0,2\pi]^2}u_s(x_1, x_2)\, d\lambda_{n,\theta}(x_1) d\lambda_{n,\theta}(x_2),
\]
and write
\begin{equation}\label{As}
A_s[G_s]:=\frac 1{4\pi^2}\int_{[0,2\pi]^2}u_s(x_1, x_2)\,dx_1dx_2.
\end{equation}
It follows from Lemma \ref{lemma:suptmu} and the bound \eqref{Unest} that
\begin{equation}\label{Unslim}
\E_n^{\beta,s}\big[|U_n^s(\theta)-A_s[G_s]|\mathbbm{1}_{E_{n,s}}\big]\to 0
\end{equation}
as $n\to\infty$.

The next lemma gives the analogous result for $R_n^s$ and $T_n^s$.

\begin{lemma} We have the limit
\[\lim_{n\to\infty} \E_n^{\beta,s}[(|T_n^s|+|R_n^s|)\mathbbm{1}_{E_{n,s}}] = 0  \]
\end{lemma}

\begin{proof}
We can choose $g_s$ on $\gamma_s$ so that $G_s(\theta)=\cos m\theta$. It follows from \eqref{expbound} that there is a constant $C(\gamma,m)$ so that    
\[
\E_n^{\beta,s}\Big[\exp\big(\sum_\mu\cos m\theta_\mu\big)\Big]\le C(\gamma,m)
\]
for all $n\ge 1$. Jensen's inequality now gives the bound
\[ 
\exp \Big( \E_n^{\beta,s}\Big[\Big|\sum_\mu \cos( m\theta_\mu)\Big|\Big] \Big) \le \E_n^{\beta,s}\Big[\exp\big(\sum_\mu\cos m\theta_\mu\big)\Big]+
\E_n^{\beta,s}\Big[\exp\big(-\sum_\mu\cos m\theta_\mu\big)\Big]\le 2C(\gamma,m),
\]
and the same bound holds for $\E_n^{\beta,s}[|\sum_\mu \sin(m\theta_\mu)|]$.
Hence there exists a strictly increasing function $f$ on $\R$ such that
\[ \E_n^{\beta,s}\Big[\Big|\sum_\mu e^{\i m\theta_\mu} \Big|\Big] \leq f(m)\]
for $m\ge 1$. From \eqref{Sumthetanu}, we also have the estimate
\[ 
\E_n^{\beta,s}\Big[\Big|\sum_\mu e^{\i m\theta_\mu} \Big|\mathbbm{1}_{E_{n,s}}\Big] \leq C(\gamma,\|G_s\|_{4,\alpha})m\sqrt{n},
\]
where $G_s$ is now the function in Theorem \ref{smainthm}.
These two bounds combined with \eqref{eq:coeffdecay}, \eqref{hkest} and \eqref{eq:Tn2} give
\begin{align*}
\E_n^{\beta,s}\Big[|T_n(\theta)| \mathbbm{1}_{E_n}\Big]  &\leq \sum_{k,l\geq 1} k|a_{kl}(s)| \E_n^{\beta,s}\Big[ \Big| h_k(s)-\frac{1}{n}\sum_{\mu=1}^n e^{-\i k\theta_\mu}H_s(\theta_\mu)\Big| \Big|\sum_{\nu=1}^n e^{-\i l\theta_\nu}\Big| \mathbbm{1}_{E_{n,s}}\Big]\\
&\leq \frac{C(\gamma,\|G_s\|_{4,\alpha}}{\sqrt{n}} \sum_{k,l\geq 1} k^2|a_{kl}| \E_n^{\beta,s}\Big[ \Big|\sum_{\nu=1}^n e^{-\i l\theta_\nu}\Big| \mathbbm{1}_{E_{n,s}}\Big]\\
&\leq\frac{C(\gamma,\|G_s\|_{4,\alpha}}{\sqrt{n}} \sum_{k\geq 1} k^2 \Big( \sum_{l=1}^{\floor{f^{-1}(n^{1/4})}} |a_{kl}|f(l) + \sqrt{n}\sum_{l> \floor{f^{-1}(n^{1/4})}} l|a_{kl}|\Big)  \\
&\le \frac{C(\gamma,\|G_s\|_{4,\alpha}}{\sqrt{n}} \sum_{k\geq 1}\frac 1{k^2}\Big(n^{1/4}\sum_{l=1}^\infty\frac 1{l^3}+\sqrt{n}\sum_{l> \floor{f^{-1}(n^{1/4})}} \frac 1{l^2}\Big)\\
&\le C(\gamma,\|G_s\|_{4,\alpha}\Big(\frac 1{n^{1/4}}+\frac 1{\floor{f^{-1}(n^{1/4})}}\Big).
\end{align*}
But since $f$ is strictly increasing so is $f^{-1}$, which is unbounded, which shows that $\E_n^{\beta,s}\Big[|T_n^s(\theta)| \mathbbm{1}_{E_{n,s}}\Big] = o(1)$ as $n\to\infty$.
The proof of $\E_n^{\beta,s}\Big[|R_n^s(\theta)| \mathbbm{1}_{E_{n,s}}\Big] = o(1)$ is similar.

\end{proof}

We now see that by \eqref{Snest}, \eqref{Unest}, Lemma \ref{lemma:Tn}, Lemma \ref{lemma:Rn} and Lemma \ref{lem:limit}
\[
\lim_{n\to\infty}\E_n^{\beta,s}\big[e^{(R_n^s+T_n^s+U_n^s+S_n^s)(\theta)}\mathbbm{1}_{E_{n,s}}\big]=e^{A_s[G_s]}.
\]
If we combine this with \eqref{IntEnsest}, we have proved
\begin{equation}\label{Aslim}
\lim_{n\to\infty}\E_n^{\beta,s}\big[\exp\Big(\sum_\mu G_s(\theta_\mu)\Big)\Big]=e^{A_s[G_s]}.
\end{equation}
It remains to show that this agrees with the formula for the limit in \eqref{eq:smain}. From \eqref{us} and \eqref{As}, we find 
\begin{align*}
A_s[G_s]&=-\frac 1{2\pi}\int_0^{2\pi}G'_s(\theta)H_s(\theta)\,d\theta-(1-\tfrac\beta 2)\frac 1{2\pi}\int_0^{2\pi}L_s'(\theta)H_s(\theta)\,d\theta\\
&-\frac{\beta}{64\pi^2}\int_0^{2\pi}\int_0^{2\pi}\frac{(H_s(\theta)-H_s(\omega))^2}{\sin^2\frac{\theta-\omega}2}\,d\theta d\omega+\frac \beta 2\Re\sum_{k,l\ge 1}kla_{kl}(s)h_k(s)h_l(s).
\end{align*}
An integration by parts gives
\begin{align}\label{Den}
-\frac{\beta}{64\pi^2}\int_0^{2\pi}\int_0^{2\pi}\frac{(H_s(\theta)-H_s(\omega))^2}{\sin^2\frac{\theta-\omega}2}\,d\theta d\omega&=
\frac{\beta}{16\pi^2}\int_0^{2\pi}\int_0^{2\pi}H_s'(\theta)(H_s(\theta)-H_s(\omega))\cot\frac{\theta-\omega}2\,d\theta d\omega\notag\\
&=-\frac{\beta}{8\pi^2}\mathrm{p.v.}\int_0^{2\pi}\int_0^{2\pi}H_s'(\theta)H_s(\omega)\frac{\partial}{\partial\omega}\log|e^{\i\theta}-e^{\i\omega}|\,d\theta d\omega
\end{align}
since $\mathrm{p.v.}\int_0^{2\pi}\cot\frac{\theta-\omega}2\,d\omega=0$. Also, by \eqref{Grunsky}
\begin{align}\label{Den2}
\frac \beta 2\Re\sum_{k,l\ge 1}kla_{kl}(s)h_k(s)h_l(s)&=\frac{\beta}{8\pi^2}\int_0^{2\pi}\int_0^{2\pi}H_s'(\theta)H'_s(\omega)\log\Big|\frac{\phi(e^{\i\theta})-\phi(e^{\i\omega})}
{e^{\i\theta}-e^{\i\omega}}\Big|\,d\theta d\omega\notag\\
&=-\frac{\beta}{8\pi^2}\int_0^{2\pi}\int_0^{2\pi}H_s'(\theta)H_s(\omega)\frac{\partial}{\partial\omega}\log\Big|\frac{\phi(e^{\i\theta})-\phi(e^{\i\omega})}
{e^{\i\theta}-e^{\i\omega}}\Big|\,d\theta d\omega.
\end{align}
Added together, \eqref{Den} and \eqref{Den2} give
\[
-\frac{\beta}{8\pi^2}\Re\int_0^{2\pi}\int_0^{2\pi}H_s'(\theta)H_s(\omega)\frac{\i e^{\i\theta}\phi'(e^{\i\theta})}{\phi(e^{\i\theta})-\phi(e^{\i\omega})}\,d\theta d\omega=
-\frac 1{4\pi}\int_0^{2\pi}G_s(\theta)H_s'(\theta)\,d\theta
\]
by \eqref{inteq}.
Thus
\[
A_s[G_s]=\frac 1{4\pi}\int_0^{2\pi}G_s(\theta)H_s'(\theta)\,d\theta+2(1-\tfrac\beta 2)\frac 1{4\pi}\int_0^{2\pi}\log|\phi'(e^{\i\theta})|H_s'(\theta)\,d\theta.
\]
Combining this formula with Lemma \ref{lem:gvar} and \eqref{Aslim}, we see that \eqref{eq:smain} follows and we have proved Theorem \ref{smainthm}, and hence also Theorem \ref{mainthm}.

\section{The partition function}\label{sec:part}

In this section we obtain the asymptotics of the partition function, 
\[ Z_{n,\beta}(\gamma)= D_n^\beta[1] = \frac{1}{n!} \int_{\gamma^n} \prod_{\mu\neq\nu}|z_\mu-z_\nu|^{\beta/2} |\d^n z|. \]
The case $\gamma=\T$ is a well-known Selberg integral and the partition function is given in \eqref{Selberg}. We will now prove the following proposition which gives the asymptotics of
the partition function for sufficiently regular curves.

\begin{prop}\label{prop:partition}
    Let $\gamma$ be a $C^{12+\alpha}$ Jordan curve, for some $\alpha>0$. Then
    \[\lim_{n\to\infty} \frac{Z_{n,\beta}(\gamma)}{Z_{n,\beta}(\T)\mathrm{cap}(\gamma)^{\beta n^2/2+(1-\beta/2)n}} = \frac{\exp \Big(\tfrac{2}{\beta}(1-\tfrac{\beta}{2})^2\dd^t(I+K)^{-1}\dd  \Big)}{\sqrt{\det(I+K)}}\]
    where $\dd$ and $K$ are given by \eqref{eq:K} and \eqref{eq:d}.
\end{prop}

This proposition together with Theorem \ref{mainthm} proves Theorem \ref{mainthm2}. In this section we will use \eqref{eq:coeffdecay} with $m=12$, i.e.~
\begin{equation}\label{eq:coeffdecay2}
    |a_{kl}|\leq A k^{-p-\epsilon}l^{-q-\epsilon}, \quad p+q=11,
\end{equation} 
for all $k,l\ge 1$.

Introduce the vectors
\begin{align*}
&\mathbf{X} = \Big(\frac{1}{\sqrt{k}}\sum_\mu \cos(k\theta_\mu)\Big)_{k\geq 1}, \quad \mathbf{Y} = \Big(\frac{1}{\sqrt{k}}\sum_\mu \sin(k\theta_\mu)\Big)_{k\geq 1}.
\end{align*}
Now, by the change of variables $z_{\mu}= \phi_s(e^{\i \theta_{\mu}})$,
\begin{align*}
Z_{n,\beta}(\gamma_s) 
&= \frac{1}{n!} \int_{[-\pi,\pi]^n} \prod_{\mu\neq\nu}|e^{\i\theta_\mu}-e^{\i\theta_\nu}|^{\beta/2} \exp\Big( \beta\sum_{1\leq \mu <\nu\leq n} \log \Big| \frac{\phi_s(e^{\i \theta_\mu})-\phi_s(e^{\i \theta_\nu})}{e^{\i\theta_\mu}-e^{\i\theta_\nu}} \Big| \\& \hskip 9.5cm + \sum_\mu \log|\phi_s'(e^{\i \theta_\mu})| \Big)\d^n\theta \\
&= \frac{1}{n!} \int_{[-\pi,\pi]^n} \prod_{\mu\neq\nu}|e^{\i\theta_\mu}-e^{\i\theta_\nu}|^{\beta/2} \exp\Big(-\frac{\beta}{2} \Re \sum_{k,l\geq 1} a_{kl}(s)\sum_\mu e^{-\i k\theta_\mu}\sum_\nu e^{-\i l\theta_\nu} \\& \hskip 9.5cm + \sum_\mu (1-\tfrac{\beta}{2})\log|\phi_s'(e^{\i \theta_\mu})| \Big)\d^n\theta 
\end{align*}
by \eqref{Grunskys}. Just as in \cite{Joh22} we can use \eqref{KGrunskyrel} and \eqref{logphiprime} and rewrite this as
\begin{equation*}
Z_{n,\beta}(\gamma_s)  = \frac{1}{n!} \int_{[-\pi,\pi]^n} \prod_{\mu\neq\nu}|e^{\i\theta_\mu}-e^{\i\theta_\nu}|^{\beta/2} \exp\Big(-\frac{\beta}{2} \begin{pmatrix}
        \mathbf{X}\\ \mathbf{Y}
    \end{pmatrix}^t K(s) \begin{pmatrix}
        \mathbf{X}\\ \mathbf{Y}
    \end{pmatrix} -2 \Big(1-\frac{\beta}{2}\Big) \dd(s)^t \begin{pmatrix}
        \mathbf{X}\\ \mathbf{Y}
    \end{pmatrix} \Big)\d^n\theta.
\end{equation*}
Differentiation gives the formula
\begin{align*}
    \frac{\d}{\d s} \log\Big(Z_n(\gamma_s)\Big) &= \frac{1}{Z_n(\gamma_s)n!} \int_{[-\pi,\pi]^n}\Big[- \tfrac{\beta}{2}\begin{pmatrix}
        \mathbf{X}\\ \mathbf{Y}
    \end{pmatrix}^t K'(s) \begin{pmatrix}
        \mathbf{X}\\ \mathbf{Y}
    \end{pmatrix} -2 \big(1-\tfrac{\beta}{2}\big) \dd'(s)^t \begin{pmatrix}
        \mathbf{X}\\ \mathbf{Y}
    \end{pmatrix}\Big] \\ 
    &\prod_{\mu\neq\nu}|e^{\i\theta_\mu}-e^{\i\theta_\nu}|^{\beta/2} \exp\Big(- \tfrac{\beta}{2}\begin{pmatrix}
        \mathbf{X}\\ \mathbf{Y}
    \end{pmatrix}^t K(s) \begin{pmatrix}
        \mathbf{X}\\ \mathbf{Y}
    \end{pmatrix} -2 \big(1-\tfrac{\beta}{2}\big) \dd(s)^t \begin{pmatrix}
        \mathbf{X}\\ \mathbf{Y}
    \end{pmatrix} \Big)\d^n\theta\\
    &= \E_n^{\beta,s}\Big[- \tfrac{\beta}{2}\begin{pmatrix}
        \mathbf{X}\\ \mathbf{Y}
    \end{pmatrix}^t K'(s) \begin{pmatrix}
        \mathbf{X}\\ \mathbf{Y}
    \end{pmatrix} -2 \big(1-\tfrac{\beta}{2}\big) \dd'(s)^t \begin{pmatrix}
        \mathbf{X}\\ \mathbf{Y}
    \end{pmatrix}\Big],
\end{align*}
and thus,
\begin{align}\label{eq:partitionlimit}
 \log \frac{Z_{n,\beta}(\gamma)}{Z_{n,\beta}(\T)}
= \int_0^1 \E_n^{\beta,s}\Big[- \tfrac{\beta}{2}\begin{pmatrix}
        \mathbf{X}\\ \mathbf{Y}
    \end{pmatrix}^t K'(s) \begin{pmatrix}
        \mathbf{X}\\ \mathbf{Y}
    \end{pmatrix} -2 \big(1-\tfrac{\beta}{2}\big) \dd'(s)^t \begin{pmatrix}
        \mathbf{X}\\ \mathbf{Y}
    \end{pmatrix}\Big] \d s.
\end{align}

The rest of this section is devoted to computing the limit of this expression as $n\to\infty$. The estimates obtained in the last section provide uniform bounds on the moments of the empirical spectral measure, and its limits, as follows.

\begin{lemma}\label{lemma:uniformbounds}
Fix $\epsilon\in (0,1)$. There is a constant $C(\gamma)$ such that
\[\sup_{s\in[0,1]} \sup_{j\geq 1}\E_n^{\beta,s}\Big[\Big(\sum_\mu\frac{\cos(j\theta_\mu)}{j^{4+\epsilon}} \Big)^2 \Big] \le C(\gamma), \quad \sup_{s\in[0,1]} \sup_{j\geq 1}\E_n^{\beta,s}\Big[\Big(\sum_\mu\frac{\sin(j\theta_\mu)}{j^{4+\epsilon}} \Big)^2 \Big] \le C(\gamma) \]
for all $n\ge 1$.
\end{lemma}

\begin{proof}
We prove the first inequality, the second one is treated similarly. 

First note that since $|\sum_\mu \cos(j\theta_\mu) |\leq n$, 
\[ \lim_{j\to\infty} \E_n^{\beta,s}\Big[\Big(\sum_\mu\frac{\cos(j\theta_\mu)}{j^{4+\epsilon}} \Big)^2 \Big] = 0,\]
so for any fixed $s\in[0,1]$ there exists a $j(n)\in\N$ (which may depend on $s$) such that
\[
\sup_{j\geq 1}\E_n^{\beta,s}\Big[\Big(\sum_\mu\frac{\cos(j\theta_\mu)}{j^{4+\epsilon}} \Big)^2 \Big] = \E_n^{\beta,s}\Big[\Big(\sum_\mu\frac{\cos(j(n)\theta_\mu)}{j(n)^{4+\epsilon}} \Big)^2 \Big].
\] 
 Set $G_n(\theta) = \frac{\cos(j(n)\theta)}{j(n)^{4+\epsilon}}$ and define the analytic function
\[ f_n^s(z) = \E_n^{\beta,s}[\exp(z\sum_{\mu}G_n(\theta_\mu))], \]
which is uniformly bounded by
\[ \E_n^{\beta,s}[\exp(\sum_\mu G_n(\theta_\mu))] + \E_n^{\beta,s}[\exp(-\sum_\mu G_n(\theta_\mu))]\]
in the closed unit disc. We see that $G_n^{(4)}(\theta)=\frac{\cos(j(n)\theta)}{j(n)^\epsilon}$, and
\[
\sup_{0\le\theta_1,\theta_2\le 2\pi}\frac {|G_n^{(4)}(\theta_1)-G_n^{(4)}(\theta_2)|}{|e^{\i\theta_1}-e^{\i\theta_2}|^\epsilon}\le 2\sup_{0\le\theta_1,\theta_2\le 2\pi}\frac 
{\big|\sin\tfrac{j(n)(\theta_1-\theta_2)}{2}\big|}{\big|j(n)\sin\tfrac{(\theta_1-\theta_2)}{2}\big|^\epsilon}\le 2\sup_x\frac{|\sin x|}{\big|j(n)\sin\tfrac{x}{j(n)}\big|^\epsilon}\le C,
\]
where $C$ is independent of $j(n)$ and hence of $s$. Hence $\|G_n\|_{4,\epsilon}<\infty$, and it follows from Theorem \ref{smainthm} that there is a constant $C(\gamma)$ such that
\[ \sup_{s\in[0,1]}\E_n^{\beta,s}[\exp(\pm \sum_\mu G_n(\theta_\mu))] \leq C(\gamma) \]
for all $n\ge 1$.
This implies that $f_n^s(z)$ is uniformly bounded by $2C(\gamma)$ in the unit disc, for all $s\in[0,1]$, so by Cauchy's estimates, $|({f_n^s})^{(k)}(0)| \leq  2k!C(\gamma)$. In particular,
\[ \sup_{j\geq 1}\E_n^{\beta,s}\Big[\Big(\sum_\mu\frac{\cos(j\theta_\mu)}{j^{4}} \Big)^2 \Big] = f_n''(0) \leq 4C(\gamma).\]
The upper bound holds for all $s\in[0,1]$ so this proves the lemma.
\end{proof}

\begin{lemma}\label{lemma:moments}
Set $P_m$ to be the projection onto the first $m$ components of $\ell_2(\R)$. For any $m\geq 1$, $s\in[0,1]$,
\begin{align*}
&\lim_{n\to\infty} \E_n^{\beta,s} \Big[ \begin{pmatrix}
        P_m & 0 \\ 0 & P_m
    \end{pmatrix}\begin{pmatrix}
        \mathbf{X}\\ \mathbf{Y}
    \end{pmatrix} \Big] = (1-\tfrac{2}{\beta})\begin{pmatrix}
        P_m & 0 \\ 0 & P_m
    \end{pmatrix}(I+K(s))^{-1}\dd(s), \\
&\lim_{n\to\infty} \E_n^{\beta,s} \Big[\begin{pmatrix}
        P_m & 0 \\ 0 & P_m
    \end{pmatrix}\begin{pmatrix}
        \mathbf{X}\\ \mathbf{Y}
    \end{pmatrix} \begin{pmatrix}
        \mathbf{X}\\ \mathbf{Y}
    \end{pmatrix}^t \begin{pmatrix}
        P_m & 0 \\ 0 & P_m
    \end{pmatrix} \Big] \\
&\qquad = \begin{pmatrix}
        P_m & 0 \\ 0 & P_m
    \end{pmatrix} \Big( \frac{1}{\beta} (I+K(s))^{-1} + (1-\tfrac{2}{\beta})^2(I+K(s))^{-1}\dd(s)\dd(s)^t(I+K(s))^{-1}\Big) \begin{pmatrix}
        P_m & 0 \\ 0 & P_m
    \end{pmatrix}.
\end{align*}
\end{lemma}

\begin{proof}
Let $z=\{ z_k \}_{k= 1}^m\in \overline{\D}^{m}$, $w=\{ w_k \}_{k= 1}^m\in \overline{\D}^{m}$. For any fixed $s$, define the analytic function in $2m$ variables,
\[f_n(z,w) = \E_n^{\beta,s}\Big[\exp\Big( \sum_{k= 1}^m \big( z_k \sum_\mu \cos(k\theta_\mu) + w_k \sum_\mu \sin (k\theta_\mu)\big) \Big)\Big].  \]
We have that
\begin{align*}
 |f_n(z,w)| &\leq \E_n^{\beta,s}\Big[\exp\Big( \sum_{k= 1}^m \big(\Re z_k \sum_\mu \cos(k\theta_\mu) + \Re w_k\sum_\mu \sin (k\theta_\mu) \big)\Big)\Big] \\
 &\leq \sum_{p \in \{-1,1\}^{2m}} \E_n^{\beta,s}\Big[\exp\Big( \sum_{k= 1}^m  \big(p_k \sum_\mu \cos(k\theta_\mu) + p_{m+k}\sum_\mu \sin (k\theta_\mu) \big)\Big)\Big]
\end{align*}
since $|z_k|\leq 1$ and $|w_k|\leq 1$ for all $1\leq k\leq m$.

By Theorem \ref{mainthm}, there exist a $N\in\N$ such that for all $p \in \{-1,1\}^{2m}$, if $n\geq N$,
\begin{align*}
    &\E_n^{\beta,s}\Big[\exp\Big( \sum_{k= 1}^m  p_k \sum_\mu \cos(k\theta_\mu) + p_{m+k} \sum_\mu \sin (k\theta_\mu) \Big)\Big] \\
&\leq 2\exp \Big( \tfrac{2}{\beta} \mathbf{v}^t(I+K(s))^{-1}\mathbf{v}+2(1-\tfrac{2}{\beta})\dd(s)^t(I+K(s))^{-1}\mathbf{v} \Big)
\end{align*}
where 
\[
\mathbf{v} = \frac{1}{2} \big((p_1,\sqrt{2}p_2,\dots,\sqrt{m}p_m,0,0,\dots)\,\, (p_{m+1},\sqrt{2}p_{m+2},\dots,\sqrt{m}p_{2m},0,0,\dots)\big)^t
\]
is a vector in $\ell_2(\R) \oplus \ell_2(\R)$. This implies that for $n\geq N$,
\[
    \big| f_n(z,w
) \big| \leq \sum_{p \in \{-1,1\}^{2m}} \exp \Big(  \|\mathbf{v}\|\|(I+K(s))^{-1}\|(\tfrac{2}{\beta}\|\mathbf{v}\| +2|1-\tfrac{2}{\beta}|\|\dd(s)\|) \Big).
\]
We know that $\sup_{s\in[0,1]} \| \dd(s) \| < \infty$ by \eqref{eq:coeffdecay2} and $\|\mathbf{v}\| = (\frac{m(m+1)}{4})^{1/2}$ for all $p$. Moreover we know that $\sup_{s\in[0,1]} \|K(s)\| \le\kappa < 1$ so $\|(I+K(s))^{-1}\| \leq (1-\kappa)^{-1}$. Hence there is a constant $C$ which does not depend on $s$ such that
\[
   \sup_{n\in\N} \big| f_n(z,w) \big| \leq C
\]
uniformly in $\overline{\D}^{2m}$, which makes the analytic functions $f_n(z,w)$ a normal family in $\D^{2m}$. By Theorem \ref{mainthm} the sequence converges pointwise, whence uniformly on compact subsets, to
\[ 
f(z,w) = \exp \Big( \tfrac{2}{\beta} \mathbf{u}^t(I+K(s))^{-1}\mathbf{u}+2(1-\tfrac{2}{\beta})\dd(s)^t(I+K(s))^{-1}\mathbf{u} \Big), 
\]
where 
\[
\mathbf{u} = \frac{1}{2} \big((z_1,\sqrt{2}z_2,\dots,\sqrt{m}z_m,0,0,\dots)\,\, (w_1,\sqrt{2}w_2,\dots,\sqrt{m}w_m,0,0,\dots)\big)^t.
\]
As a consequence its derivatives converge to those of $f(z,w)$. In particular, 
\[ \lim_{n\to\infty} \E_n^{\beta,s} \Big[ \sum_\mu \cos(j\theta_\mu) \Big] = \frac{\partial}{\partial z_j}f(z,w)\Big|_{z=w=0} = \sqrt{j} (1-\tfrac{2}{\beta})((I+K(s))^{-1}\dd(s))_j. \]
The second derivatives give the limiting variances and covariances. For example,
\begin{align*}
    &\lim_{n\to\infty} \E_n^{\beta,s} \Big[ \sum_\mu \cos(j\theta_\mu) \sum_\mu \sin(k\theta_\mu) \Big] = \frac{\partial}{\partial z_j} \frac{\partial}{\partial w_k} f(z,w)\Big|_{z=w=0} \\
    &= \frac{\sqrt{jk}}{\beta}(I+K(s))^{-1}_{j,m+k} +\sqrt{jk} (1-\tfrac{2}{\beta})^2((I+K(s))^{-1}\dd(s))_j((I+K(s))^{-1}\dd(s))_{m+k},
\end{align*}
and this proves the lemma.

\end{proof}

We are now ready for the proof of the proposition on the asymptotics of the partition function.
\begin{proof}[Proof of Proposition \ref{prop:partition}.]
Proposition \ref{lemma:uniformbounds} allows us to use the dominated convergence theorem to compute the limit \eqref{eq:partitionlimit}. Indeed, by the definition of $K(s)$ and $d(s)$,
\begin{align*}
    & \sup_{s\in[0,1]}\Big| \E_n^{\beta,s}\Big[ \begin{pmatrix}
        \mathbf{X}\\ \mathbf{Y}
    \end{pmatrix}^t K'(s) \begin{pmatrix}
        \mathbf{X}\\ \mathbf{Y}
    \end{pmatrix} \Big] \Big| \leq  \sup_{s\in[0,1]} \sum_{k,l\geq 1} \frac{1}{\sqrt{kl}}(k+l)s^{k+l-1} \Big(  |b_{kl}^{(1)}(s)| \E_n^{\beta,s} \Big| \sum_\mu \cos(k\theta_\mu) \sum_\nu \cos(l\theta_\nu) \Big| \\ &+ 2|b_{kl}^{(2)}(s)| \E_n^{\beta,s} \Big| \sum_\mu \cos(k\theta_\mu) \sum_\nu \sin(l\theta_\nu)\Big| + |b_{kl}^{(1)}(s)| \E_n^{\beta,s} \Big| \sum_\mu \sin(k\theta_\mu) \sum_\nu \sin(l\theta_\nu)\Big|  \Big) \\
    &\leq C(\gamma)\sum_{k,l\geq 1} (k+l)|a_{kl}|(kl)^{4+\epsilon}, 
\end{align*}
where we used the Cauchy-Schwarz inequality and Lemma \ref{lemma:uniformbounds}. By \eqref{eq:coeffdecay2} this is bounded for $\epsilon$ sufficiently small. Similarly,
\begin{align*}
    \sup_{s\in[0,1]}\Big| \E_n^{\beta,s}\Big[ \dd'(s)^t \begin{pmatrix}
        \mathbf{X}\\ \mathbf{Y}
    \end{pmatrix}\Big] \Big| &\leq \sup_{s\in[0,1]} \sum_{k\geq 1} \sqrt{k} \sum_{j=1}^{k-1} |a_{j,k-j}|\Big(\E_n^{\beta,s} \Big| \sum_\mu \cos(k\theta_\mu)\Big|+ \E_n^{\beta,s} \Big| \sum_\mu \sin(k\theta_\mu)\Big| \Big) \\ &<\infty
\end{align*} 
by Lemma \ref{lemma:uniformbounds} and \eqref{eq:coeffdecay2}.
Thus,
\begin{align*}
 &\lim_{n\to\infty} \int_0^1 \E_n^{\beta,s}\Big[- \frac{\beta}{2}\begin{pmatrix}
        \mathbf{X}\\ \mathbf{Y}
    \end{pmatrix}^t K'(s) \begin{pmatrix}
        \mathbf{X}\\ \mathbf{Y}
    \end{pmatrix} -2 \Big(1-\frac{\beta}{2}\Big) \dd'(s)^t \begin{pmatrix}
        \mathbf{X}\\ \mathbf{Y}
    \end{pmatrix}\Big] \d s \\
    &= \int_0^1 \lim_{n\to\infty} \E_n^{\beta,s}\Big[- \frac{\beta}{2}\begin{pmatrix}
        \mathbf{X}\\ \mathbf{Y}
    \end{pmatrix}^t K'(s) \begin{pmatrix}
        \mathbf{X}\\ \mathbf{Y}
    \end{pmatrix} -2 \Big(1-\frac{\beta}{2}\Big) \dd'(s)^t \begin{pmatrix}
        \mathbf{X}\\ \mathbf{Y}
    \end{pmatrix}\Big] \d s
\end{align*}
Moreover, for all $n$ large enough,
\begin{multline*}
    \Big| \E_n^{\beta,s}\Big[ \begin{pmatrix}
        \mathbf{X}\\ \mathbf{Y}
    \end{pmatrix}^t K'(s) \begin{pmatrix}
        \mathbf{X}\\ \mathbf{Y}
    \end{pmatrix} \Big] -  \E_n^{\beta,s}\Big[ \begin{pmatrix}
        \mathbf{X}\\ \mathbf{Y}
    \end{pmatrix}^t \begin{pmatrix}
        P_m & 0 \\ 0 & P_m
    \end{pmatrix} K'(s) \begin{pmatrix}
        P_m & 0 \\ 0 & P_m
    \end{pmatrix}\begin{pmatrix}
        \mathbf{X}\\ \mathbf{Y}
    \end{pmatrix} \Big] \Big| \\
    \leq \sum_{\max(k,l)\geq m} \frac{1}{\sqrt{kl}}(k+l)s^{k+l-1} \Big(  |b_{kl}^{(1)}| \E_n^{\beta,s} \Big| \sum_\mu \cos(k\theta_\mu) \sum_\nu \cos(l\theta_\nu) \Big| + 2|b_{kl}^{(2)}| \E_n^{\beta,s} \Big| \sum_\mu \cos(k\theta_\mu) \sum_\nu \sin(l\theta_\nu)\Big| \\ +|b_{kl}^{(1)}| \E_n^{\beta,s} \Big| \sum_\mu \sin(k\theta_\mu) \sum_\nu \sin(l\theta_\nu)\Big|  \Big) \leq C(\gamma)\sum_{k\geq 1, l\geq m} (kl)^{-1-\epsilon} = \frac{C(\gamma)}{m^{\epsilon}}.
\end{multline*}
The bound is uniform in $n$, hence,
\[ \lim_{n\to\infty}  \E_n^{\beta,s}\Big[ \begin{pmatrix}
        \mathbf{X}\\ \mathbf{Y}
    \end{pmatrix}^t K'(s) \begin{pmatrix}
        \mathbf{X}\\ \mathbf{Y}
    \end{pmatrix} \Big] = \lim_{m\to\infty} \lim_{n\to\infty}  \E_n^{\beta,s}\Big[ \begin{pmatrix}
        \mathbf{X}\\ \mathbf{Y}
    \end{pmatrix}^t \begin{pmatrix}
        P_m & 0 \\ 0 & P_m
    \end{pmatrix} K'(s) \begin{pmatrix}
        P_m & 0 \\ 0 & P_m
    \end{pmatrix}\begin{pmatrix}
        \mathbf{X}\\ \mathbf{Y}
    \end{pmatrix} \Big]. \]
The limit can now be computed with Lemma \ref{lemma:moments}:
\begin{align*}
     =\frac{1}{\beta} \mathrm{Tr}\ K'(s)(I+K(s))^{-1}+(1-\tfrac{2}{\beta})^2 \dd(s)^t(I+K(s))^{-1}K'(s)(I+K(s))^{-1}\dd(s).
\end{align*}
Similarly,
\[\lim_{n\to\infty }\E_n^{\beta,s}\Big[ \dd'(s)^t \begin{pmatrix}
        \mathbf{X}\\ \mathbf{Y}
    \end{pmatrix}\Big] = \lim_{m\to\infty} \lim_{n\to\infty }\E_n^{\beta,s}\Big[ \dd'(s)^t \begin{pmatrix}
        P_m & 0 \\ 0 & P_m
    \end{pmatrix}\begin{pmatrix}
        \mathbf{X}\\ \mathbf{Y}
    \end{pmatrix}\Big] = (1-\tfrac{2}{\beta}) \dd'(s)^t (I+K(s))^{-1}\dd(s). \]
Combined, these limits give
\begin{gather*}
\lim_{n\to\infty} \E_n^{\beta,s}\Big[- \frac{\beta}{2}\begin{pmatrix}
        \mathbf{X}\\ \mathbf{Y}
    \end{pmatrix}^t K'(s) \begin{pmatrix}
        \mathbf{X}\\ \mathbf{Y}
    \end{pmatrix} -2 \Big(1-\frac{\beta}{2}\Big) \dd'(s)^t \begin{pmatrix}
        \mathbf{X}\\ \mathbf{Y}
    \end{pmatrix}\Big] = 2(1-\tfrac{2}{\beta})^2 \dd'(s)^t (I+K(s))^{-1}\dd(s) \\
 -\frac{1}{2} \mathrm{Tr}\ K'(s)(I+K(s))^{-1}-\tfrac{\beta}{2}(1-\tfrac{2}{\beta})^2 \dd(s)^t(I+K(s))^{-1}K'(s)(I+K(s))^{-1}\dd(s) \\
=  \frac{\d}{\d s} \Big( -\tfrac{1}{2} \log \det (I+K(s)) + \tfrac{\beta}{2}(1-\tfrac{2}{\beta})^2 \dd(s)^t(I+K(s))^{-1}\dd(s)\Big).    
\end{gather*}
Integrating from $0$ to $1$, and using that $K(0)=d(0)=0$ and $K(1)=K$, $d(1)=d$, finishes the proof.
\end{proof}

%\section{Proof of Proposition \ref{prop:diffvar}}

\end{document}